\pdfoutput=1
\RequirePackage{fix-cm}
\documentclass[abstract]{scrartcl}
\setkomafont{title}{}
\usepackage{filecontents}

\usepackage[T1]{fontenc}
\usepackage[utf8]{inputenc}
\usepackage{lmodern,microtype,fixltx2e}
\usepackage[export,calc]{adjustbox}
\usepackage{graphicx}
\let\theoremstyle\undefined
 
\usepackage{amsmath,amssymb,amsthm,tabu,enumitem}
\usepackage{mathtools,etoolbox,fncylab,suffix}
\usepackage[colorlinks]{hyperref}
\usepackage[lite,alphabetic]{amsrefs}

\title{Obvious natural morphisms of sheaves are unique}
\author{Ryan Cohen Reich\thanks{email:~\nolinkurl{ryancr@umich.edu}. This
    research was partially supported by NSF RTG grant DMS 0943832.}}
\newcommand*\includestandalone{\includegraphics[valign=M]}

\hypersetup{
 linkcolor = blue,
 urlcolor = blue
}

\setlist[itemize,enumerate]{labelindent = 0.5em, leftmargin=*}
\setlist[description]{labelindent = 0.5em, itemindent = 0.5em,
  leftmargin = *}

\labelformat{section}{\rrplaintest{Section~}#1}
\labelformat{figure}{\rrplaintest{Figure~}#1}
\numberwithin{equation}{section}

\makeatletter
\renewcommand{\BibLabel}{%
  \Hy@raisedlink{\hyper@anchorstart{cite.\CurrentBib}\relax\hyper@anchorend}%
  [\thebib]%
}
\makeatother

\BibSpec{online}{%
  +{}{\PrintAuthors} {author}
  +{,}{ \textit} {title}
  +{:}{ \textit} {subtitle}
  +{}{ \parenthesize} {date}
  +{,}{ \url} {url}
  +{.}{ Accessed } {accessdate}
  +{.}{ } {note}
  +{.}{ } {transition}
}

\newcommand*{\genby}[1]{%
  \begingroup
    \begingroup
      \lccode`\~=`\,
    \lowercase{\endgroup\let~\cup}%
    \mathcode`\,=\string"8000
    \langle#1\rangle
  \endgroup
}
\DeclareMathOperator{\roof}{roof}
\DeclareMathOperator{\source}{S}
\DeclareMathOperator{\target}{T}
\DeclareMathOperator{\comp}{\nattrans{comp}}
\DeclareMathOperator{\triv}{\nattrans{triv}}
\DeclareMathOperator\SGNT{\nattrans{SGNT}}
\newcommand*\cSGNT{\SGNT}
\DeclareMathOperator\bc{\nattrans{bc}}
\DeclareMathOperator\cd{\nattrans{cd}}

\DeclareMathOperator\RNA{\nattrans{RNA}}
\newcommand*\define{\emph}
\newcommand*\isoarrow{\xrightarrow{\sim}}
\newcommand{\nattrans}[1]{\mathsf{#1}}

\makeatletter
\newcommand*\unit{
  \@ifnextchar'{\on{\nattrans{counit}}\@gobble}{\on{\nattrans{unit}}}%
}
\newcommand*\Unit{
  \@ifnextchar'{\on{\nattrans{Counit}}\@gobble}{\on{\nattrans{Unit}}}%
}
\makeatother
\newcommand*\sh{\mathcal}
\newcommand*\cat{\mathbf}
\newcommand*\on{\operatorname}

\newcommand*\R{\mathbb{R}}
\newcommand*\id{\mathrm{id}}
\newcommand*\op{\mathrm{op}}
\newcommand*\Aff{\mathbb{A}}

\makeatletter
\newcommand{\clonecounter}[2]{
 \@xp\xdef\csname c@#1\endcsname{\@xp\@nx\csname c@#2\endcsname}
 \@xp\xdef\csname the#1\endcsname{\@xp\@nx\csname the#2\endcsname}
 \@xp\xdef\csname theH#1\endcsname{\@xp\@nx\csname theH#2\endcsname}
 \global\@xp\let\csname p@#1\endcsname = \@empty
 \global\@xp\let\csname cl@#1\endcsname = \@empty
}
\def\@ynthm#1[#2]#3{%
  \ifx\relax#2\relax
  \def\@tempa{\@oparg{\@xthm{#1}{#3}}[]}%
  \else
  \@ifundefined{c@#2}{%
    \def\@tempa{\@nocounterr{#2}}%
  }{%
    \clonecounter{#1}{#2}%
    \toks@{#3}%
    \@xp\xdef\csname#1\endcsname{%
      \@nx\@thm{%
        \let\@nx\thm@swap
        \if S\thm@swap\@nx\@firstoftwo\else\@nx\@gobble\fi
        \@xp\@nx\csname th@\the\thm@style\endcsname}%
      {#1}{\the\toks@}}% I changed #2 to #1 here
    \let\@tempa\relax
  }%
  \fi
  \@tempa
}
\makeatother

\newenvironment{envbyname}[1]
{%
  \def\environmentname{#1}%
  \csname\environmentname\endcsname
}
{%
  \csname end\environmentname\endcsname
}

\newenvironment{theorem}[2]
{\envbyname{#1}\label{#2}}
{\endenvbyname}

\newenvironment{theorem*}[1]
{\envbyname{#1*}}
{\endenvbyname}

\numberwithin{thmcounter}{section}

\newcommand*\createtheorem[2]{%
  \newtheorem{#1}[thmcounter]{#2}%
  \newtheorem*{#1*}{#2}%
  \labelformat{#1}{\rrplaintest{#2~}##1}%
}

\newif\ifrrplainref % Starts false
\DeclareRobustCommand{\rrplaintest}[1]{\ifrrplainref\else #1\fi}
\newcommand{\plainref}[1]{\rrplainreftrue\ref{#1}\rrplainreffalse}
\WithSuffix\newcommand\plainref*[1]{\rrplainreftrue\ref*{#1}\rrplainreffalse}

\theoremstyle{definition}
\createtheorem{thm}{Theorem}
\createtheorem{defn}{Definition}
\createtheorem{lem}{Lemma}
\createtheorem{prop}{Proposition}
\createtheorem{cor}{Corollary}
\createtheorem{ex}{Example}
\createtheorem{exer}{Exercise}
\createtheorem{prb}{Problem}
\labelformat{equation}{(#1)}
\labelformat{enumi}{(#1)}

\begin{document}
\maketitle
\begin{abstract}\noindent
  We prove that a large class of natural transformations (consisting
  roughly of those constructed via composition from the ``functorial''
  or ``base change'' transformations) between two functors of the form
  $\cdots f^* g_* \cdots$ actually has only one element, and thus that
  any diagram of such maps necessarily commutes.  We identify the
  precise axioms defining what we call a ``geofibered category'' that
  ensure that such a coherence theorem exists.  Our results apply to
  all the usual sheaf-theoretic contexts of algebraic geometry.  The
  analogous result that would include any other of the six functors
  remains unknown.
\end{abstract}

Commutative diagrams express one of the most typical subtle beauties
of mathematics, namely that a single object (in this case an arrow in
some category) can be realized by several independent constructions.
The more interesting the constructions, the more insight is gained by
carefully verifying commutativity, and it is tempting to take the
inverse claim to mean that comparable arrows, constructed mundanely,
should be expected to be equal and are therefore barely worth proving
to be so.  As half the purpose of a human-produced publication is to
provide insight, there is some validity in thinking that any
opportunity to reduce both length and tedium is worth taking, but this
is at odds with the other half, which is to supply rigor.  The goal of
this paper is, therefore, to serve both ends by proving that a large
class of diagrams commonly encountered in algebraic geometry, obtained
from the fibered category nature of sheaves, are both interesting and
necessarily commutative.

The goal of proving ``all diagrams commute'' began with the first
``coherence theorem'' of this kind, proved by Mac~Lane~\cite{maclane};
further research, apparently, has not yet developed this particular
application, as we are not aware of any coherence theorems that apply
to fibered categories.  Indeed, our main result~\ref{main theorem} is
itself not entirely free of conditions, and our best unconditional
result such as in \ref{secondary theorem} is somewhat restricted in
scope.  A similar statement to the latter was obtained by Jacob
Lurie~\cite{lurie} concerning \emph{arbitrary} tensor isomorphisms of
functors on coherent sheaves over a stack; ours applies to less
general morphisms but more general situations.  The problem of proving
a general coherence theorem for pullbacks and pushforwards in the
context of \emph{monoidal} functors was mentioned in
Fausk--Hu--May~\cite{fausk-may} and said to be both unsolved and
desirable; we do not claim to have solved it, though certainly, we
have done something.  Those authors do not consider multiple maps~$f$
in combination, however, as we do.

The goal of rigoriously verifying the diagrams of algebraic geometry
has been pursued by several people, notably recently Brian Conrad
in~\cite{conrad}, who proved the compatibility of the trace map of
Grothendieck duality with base change.  Our theorems do not seem to
apply to this problem, most significantly because almost every
construction in that book intimately refers to details of abelian and
derived categories and to methods of homological algebra, none of
which are addressed here.  Such a quality is possessed by many
constructions of algebraic geometry and place many of its interesting
compatibilities potentially out of reach of the results stated here.

Maps constructed in such a way, as concerns Conrad constantly in that
book, can certainly differ by a natural sign, if not worse, and so the
only hope for automatically proving commutativity of that kind of
diagram is to separate the cohomological part from the categorical
part and apply our results only to the latter.  This is to say nothing
of the explicit avoidance of ``good'' hypotheses (for example,
flatness, such as we will consider) on the schemes and morphisms
considered there, all of which place this particular compatibility on
the other side of the line separating ``interesting'' from
``mundane''.

In \ref{s:definitions}, we give (many) definitions in preparation for
stating the main theorems; comments on the hypotheses considered there
are given in \ref{s:comments}, as are the acknowledgements.  These
theorems are given, without proof, in \ref{sec:main thms}, and their
proofs are delayed until \ref{sec:proofs}.  The style of these first
two sections of this paper is quite formal; for a more relaxed
overview, see \ref{sec:guide}.  \ref{s:string diagrams} quickly
presents our use of ``string diagrams'', a computational tool we will
use to give structure to some of our more arbitrary calculations;
further details and comparison with others' use of this tool is given
in \ref{s:string diagram intro}. Afterwards,
Sections~\plainref{sec:uniqueness} and~\plainref{sec:simplify} contain
the core arguments upon which the eventual high-level proofs rely.

\section{Definitions and notation}
\label{s:definitions}

In this section we define the natural transformations we consider and
will eventually prove are unique.  Since they apply to a variety of
similar but slightly different common situations in algebraic
geometry, we have abstracted the essential properties into a formal
object of category theory; the particular applications are indicated
in \ref{sec:main thms}.

\subsection*{Abstract push and pull functors.}
Our results apply in a variety of similar situations with slightly
different features and technical hypotheses, all unified by the
formalism of pushforward and pullback.  The following definition
encapsulates the axiomatic properties we will require.

\begin{theorem}{defn}{geofibered category}
  We define a \define{geofibered category} to be a functor $F \colon
  \cat{Sh} \to \cat{Sp}$ between two categories called, respectively,
  ``shapes'' and ``spaces'', such that:
  \begin{itemize}
  \item The category~$\cat{Sp}$ has all finite fibered products.
  \item For every morphism $f \colon X \to Y$ of~$\cat{Sp}$, there are
    functors
    \begin{align*}
      f^* \colon \cat{Sh}_Y \longleftrightarrow \cat{Sh}_X : f_*
    \end{align*}
    between the fibers over~$Y$ and~$X$, which are an adjoint pair
    $(f^*, f_*)$.
  \item The assignment $X \mapsto \cat{Sh}_X$ and $f \mapsto f^*$
    constitutes a pseudofunctor $\cat{Sp} \to \cat{Cat}$; i.e.\ a
    cleaved fibered category, as described below.
  \item Dually, the assignment $X \mapsto \cat{Sh}_X$ and $f \mapsto
    f_*$ constitutes a pseudofunctor $\cat{Sp} \to \cat{Cat}^\op$.
  \item The above pseudofunctor data obey the compatibilities
    described in the remainder of this subsection.
  \end{itemize}
\end{theorem}

The terminology was chosen in part because ``bifibered category'' has
a different meaning, and also to indicate that this concept models the
fibered category of sheaves in various geometries.  For a description
of pseudofunctors in the context of fibered categories, see Vistoli's
exposition~\cite{vistoli}*{\S3.1.2}.

The requirement that~$\cat{Sp}$ has fibered products is only strictly
speaking necessary for some of the following definitions, but since
that of the ``roof'' is among them, given the central nature of this
concept in our main theorems, a geofibered category would be quite
useless otherwise.

The rest of this subsection is devoted to introducing notation and
terminology and, along side it, describing the compatibilities of the
data of a geofibered category.

\begin{theorem}{defn}{SGFs}
  As per \ref{geofibered category}, in any geofibered category, for
  any morphism $f \colon X \to Y$ of spaces we have adjoint functors
  $(f^*, f_*)$ of shapes which we call \define{basic standard
    geometric functors} (the terminology intentionally references
  ``geometric morphisms'' of topoi)

  We define a \define{standard geometric functor} (SGF) to be any
  composition of basic standard geometric functors.  Formally, an~SGF
  is equivalent to a diagram of spaces and morphisms forming a
  directed graph that is topologically linear, together with an
  ordering of its vertices from one end of the segment to the other
  (i.e.~we ``remember'' the terms of the composition as well as the
  resulting functor).  For an~SGF~$F$ from shapes on~$X$ to shapes
  on~$Y$, we write~$X = \source(F)$ and~$Y = \target(F)$.
\end{theorem}

The pseudofunctor data consists of a number of natural
transformations, which are the basis for the main objects of our
study.

\begin{theorem}{defn}{basic SGNTs}
  Between pairs of~SGFs there are canonical natural transformations of
  the following three types that we call the \define{basic standard
    geometric natural transformations}:
  \begin{subequations}
    \begin{align}
      \label{eq:units}
      \left.
      \begin{aligned}[c]
        \unit(f) &\colon \on{id} \to f_* f^* \\
        \unit'(f) &\colon f^* f_* \to \on{id}
      \end{aligned}
      \right\}&
      \text{\define{Adjunctions}} \\
      \label{eq:compositions}
      \left.
      \begin{aligned}[c]
        \comp^*(f,g) &\colon f^* g^* \isoarrow (gf)^* \\
        \comp_*(f,g) &\colon g_* f_* \isoarrow (gf)_*
      \end{aligned}
      \right\} &
      \text{\define{Compositions}} \\
      \label{eq:trivializations}
      \left.
      \begin{aligned}[c]
        \triv_*(f) &\colon \id_* \isoarrow \id \\
        \triv^*(f) &\colon \id^* \isoarrow \id
      \end{aligned}
      \right\} &
      \text{\define{Trivializations}}.
    \end{align}
  \end{subequations}
  For the latter two types, we will use~$^{-1}$ to denote their
  inverses (which will not arise as often in our arguments).
\end{theorem}

The origins, nature, and compatibilities of these transformations are
described in the following subpoints.

\subsubsection*{Adjunctions.}

% The required adjunctions that are part of the definition of a
% geofibered category yield unit and counit morphisms.  For precision,
% and for later rigorous use, we state their exact relationships to the
% adjunction bijection.  Denoting the adjunction bijections
% (independently of particular shapes $\sh{F}$~and~$\sh{G}$)
% \begin{equation}
%   \label{eq:adjunction}
%   \ADJ_f \colon \on{Hom}(f^* \sh{F}, \sh{G})
%   \longleftrightarrow \on{Hom}(\sh{F}, f_* \sh{G}),
% \end{equation}
% they define and are defined by the unit and counit maps
% \begin{subequations}
%   \label{eq:adjunction definitions}
%   \begin{align}
%     \label{eq:adjunction forward}
%     \ADJ_f(f^* \sh{F} \xrightarrow{\phi} \sh{G}) 
%     &= f_* \phi \circ \unit(f)_{\sh{F}} &
%     \unit(f)_{\sh{F}} &= \ADJ_f(\id \colon f^*\sh{F} \to f^*\sh{F}) \\
%     \label{adjunction reverse}
%     \ADJ_f(\sh{F} \xrightarrow{\psi} f_* \sh{G})
%     &= \unit'(f)_{\sh{G}} \circ f^* \psi &
%     \unit'(f)_{\sh{G}} &= \ADJ_f(\id \colon f_*\sh{G} \to f_*\sh{G}).
%   \end{align}
% \end{subequations}
The maps $\unit(f)$~and~$\unit'(f)$ are equivalent to the~$(f^*, f_*)$
adjunction in the usual way and satisfy the familiar compatibility
required of an adjunction, which we state at risk of being pedantic so
as to provide a complete reference for the data of a geofibered
category.
\begin{subequations}
  \label{eq:adjoint identities}
  \begin{alignat}{99}
    \label{eq:left adjoint identity}
    \bigl(
    f^*
    \xrightarrow{f^* \unit(f)}
    &f^* f_* f^*
    \xrightarrow{\unit'(f) f^*}
    f^*
    \bigr) &
    &= \id \\
    \label{eq:right adjoint identity}
    \bigl(
    f_*
    \xrightarrow{\unit(f)f_*}
    &f_* f^* f_*
    \xrightarrow{f_* \unit'(f)}
    f_*
    \bigr)&
    &= \id
  \end{alignat}
\end{subequations}

In the context of \emph{natural transformations of functors} rather
than morphisms of individual objects, we have a ``reverse natural
adjunction'' also defined by the unit and counit; we use it on
occasion to aid definitions.  The construction and proof of
bijectivity are left as an exercise for readers who desire it.

\begin{equation}
  \label{eq:RNA map}
  \RNA_f \colon \on{Hom}(Ff_*, G) \longleftrightarrow \on{Hom}(F, Gf^*).
\end{equation}

\subsubsection*{Compositions.}
The data of a pseudofunctor implies the existence of
isomorphisms $\comp_*(f,g) \colon (fg)_* \cong f_* g_*$ for every
composable pair of morphisms~$f,g$ of spaces.  Likewise, we have
isomorphisms $\comp^*(f,g) \colon (fg)^* \cong g^* f^*$.  The
compatibilities required of these data by a pseudofunctor express
their associativity in triple compositions:
\begin{equation}
  \label{eq:compositions assoc}
  \bigl(
  (fgh)_* \cong (fg)_* h_* \cong f_* g_* h_*
  \bigr)
  =
  \bigl(
  (fgh)_* \cong f_* (gh)_* \cong f_* g_* h_*
  \bigr)
\end{equation}
and similarly for pullbacks. It must be noted that this data for
pushforwards or pullbacks alone determines such data for the other;
for example, given pseudofunctoriality of $f_*$, for the
adjoint~$f^*$, we define~$\comp^*(f,g)$ via category-theoretic
formalism:~$(fg)_*$ has the left adjoint~$(fg)^*$ and the
composition~$f_* g_*$ has as left adjoint the composition~$g^* f^*$;
since we have~$(fg)_* \cong f_* g_*$ we also get~$(fg)^* \cong g^*
f^*$ by uniqueness of left adjoints.  In more explicit terms, this
means that for any shapes $\sh{F}$~and~$\sh{G}$, we have an isomorphism
\begin{multline}
  \label{eq:composition adjunction}
  \on{Hom}((fg)^* \sh{F}, \sh{G})
  \cong
  \on{Hom}(\sh{F}, (fg)_* \sh{G})
  \cong
  \on{Hom}(\sh{F}, f_* g_* \sh{G})
  \cong
  \on{Hom}(f^* \sh{F}, g_* \sh{G}) \\
  \cong
  \on{Hom}(g^* f^* \sh{F}, \sh{G}).
\end{multline}
We require as a compatibility relation that this is
$\on{Hom}(\comp^*(f,g)_{\sh{F}}, \sh{G})$.

\subsubsection*{Trivializations.}
The trivialization isomorphism~$\triv_* \colon \id_* \cong \id$ is
another part of the pseudofunctor data (and likewise for pullbacks),
as well as its compatibility with composition:
\begin{equation}
  \label{eq:trivializations comp}
  \bigl(\comp_*(\id, f) \colon \id_* f_* \cong (\id f)_* = f_*\bigr)
  = \bigl(\triv_* f_* \colon \id_* f_* \cong \id f_* = f_* \bigr)
\end{equation}
(and, again, the same for pullbacks).  Given this data just for
pullbacks or pushforwards, for example the latter, we could define an
isomorphism $\id^* \to \id$ and its inverse by adjunction,
% \begin{align}
%   \label{eq:trivializations def}
%   \ADJ_{\id}(\triv_* \colon \id \to \id_*) &&
%   \RNA_{\id}(\triv_* \colon \id_* \to \id).
% \end{align}
% We defer until later the proof that they are inverse; it is based on
% the following isomorphism,
similar to~\ref{eq:composition adjunction}:
\begin{equation}
  \label{eq:trivialization adjunction}
  \on{Hom}(\id^* \sh{F}, \sh{G})
  \cong
  \on{Hom}(\sh{F}, \id_* \sh{G})
  \cong
  \on{Hom}(\sh{F}, \sh{G}).
\end{equation}
We require that this isomorphism be equal
to~$\on{Hom}(\triv^*_{\sh{F}}, \sh{G})$.

\subsection*{Classes of transformations.}
Typically we do not consider transformations in the pure form above,
but rather in horizontal composition with some identity maps.

\subsubsection*{Basic classes.}
The following are the classes of natural transformations consisting of
just one of the basic ones in horizontal composition.

\begin{theorem}{defn}{basic SGNT classes}
  In this definition, $F$~and~$G$ represent any~SGFs; $f$~and~$g$
  represent any maps of spaces.  We define:
  \begin{subequations}
    \begin{align}
      \label{eq:unit class}
      \unit &= \{F\unit(f)G\} &
      \unit' &= \{F\unit'(f)G\} \\
      \label{eq:composition class}
      \comp_0 &= \{F\comp_*(f,g)G\} \cup \{F \comp^*(f,g) G\} &
      \comp &= \comp_0 \cup \comp_0^{-1} \\
      \label{eq:triv class}
      \triv_0 &= \{F\triv_*G\} \cup \{F \triv^* G\} &
      \triv &= \triv_0 \cup \triv_0^{-1}.
    \end{align}
  \end{subequations}
\end{theorem}

And now we may introduce the main objects of our study.

\begin{theorem}{defn}{SGNTs}
  A \define{standard geometric natural transformation} (SGNT) is an
  element of the class (in which we use the notation $\genby{S}$
  to denote the class generated by $S$ via composition)
  \begin{equation}
    \label{eq:SGNT class}
    \SGNT = \genby{\unit \cup \unit' \cup \comp \cup \triv}.
  \end{equation}
  For any SGNT~$\phi \colon F \to G$ between two~SGFs, we write~$F =
  \on{dom} \phi$ and~$G = \on{cod} \phi$.
\end{theorem}

Another way of understanding this class is the following
characterization:~$\SGNT$ is the smallest category of natural
transformations of~SGFs containing all identity maps,~$\comp$
maps,~$\triv$ maps,~and their inverses, and that is closed under
horizontal and vertical composition and adjunction of~$^*$ and~$_*$.
% (see the equations~\ref{eq:adjunction definitions} for an
% explanation).

The following type of~SGNT is of fundamental importance throughout the
paper; it arises seemingly of its own accord in a variety of
situations and is also essential in simplifying~SGFs to the point that
our main theorem is provable.

\begin{theorem}{defn}{commutative diagram morphism}
  Consider a commutative diagram
  \begin{equation}
    \label{eq:base change diagram}
    \includestandalone{img0}
  \end{equation}
  We define the associated \define{commutation morphism}
  \begin{equation}
    \label{eq:base change morphism} 
    \cd(f,g;\tilde{f},\tilde{g}) \colon g^* f_* \to \tilde{f}_* \tilde{g}^*
  \end{equation}
  to be the map corresponding by $g$-adjunction and $\tilde{g}$-
  reverse natural adjunction to the corner-swapping map
  \begin{equation}
    \label{eq:base change definition}
    f_* \tilde{g}_*
    \xrightarrow{\comp_*(\tilde{g},f)}
    (f\tilde{g})_*
    = (g\tilde{f}_*)
    \xrightarrow{\comp_*(\tilde{f},g)^{-1}}
    g_* \tilde{f}_* 
  \end{equation}
  See~\ref{eq:CD string diagram} for a visualization of this
  definition.
\end{theorem}

The commutation morphisms are too general to be entirely useful,
though it is helpful to retain the concept for manipulations when no
further hypotheses are needed.  Nonetheless, for our theorems such
hypotheses are needed.

\begin{theorem}{defn}{base change morphism}
  When~\ref{eq:base change diagram} is cartesian, we write~$\bc(f,g) =
  \cd(f,g;\tilde{f},\tilde{g})$ and use the aggregate notation, as
  before:
  \begin{equation}
    \label{eq:base change class}
    \bc = \{F\bc(f,g)G\}.
  \end{equation}
\end{theorem}

\subsubsection*{Invertible base changes.}
In this subsection and the next, we introduce some further conditions
on geofibered categories that support a larger class of natural
transformations to which to apply our theorems.  We include the proofs
of a few of their simplest properties, some of which are necessary for
the definitions to make sense and others are simply most appropriate
when placed here.

We begin with an abstract structure on geofibered categories inspired
by \ref{base change morphism}.

\begin{theorem}{defn}{geolocalizing}
  Let $\cat{Sh} \to \cat{Sp}$ be a geofibered category.  We will call
  a class~$P$ of morphisms in~$\cat{Sp}$ \define{push-geolocalizing}
  if:
  \begin{itemize}
  \item It contains every isomorphism and is closed under composition.
  \item For each~$f \in P$ and~$g \in \cat{Sp}$, the base
    change~$\bc(f,g)$ is an isomorphism and~$\tilde{f} \in P$.
  \end{itemize}
  Similarly, we define~$P$ to be \define{pull-geolocalizing} by
  swapping the roles of~$f$ and~$g$. We say that a geofibered category
  is given a~\define{geolocalizing} structure if it is equipped with a
  pair of push-~and pull-geolocalizing classes in~$\cat{Sp}$.
\end{theorem}

We chose the word ``localizing'' in deference to the convention that
to invert morphisms of a category is to localize it; the full term
``geolocalizing'' is in part to complement ``geofibered'', and in part
to indicate that it is not the geolocalizing morphisms themselves that
are inverted.

We note that, by definition, every geofibered category has a
``trivial'' geolocalizing structure consisting of just the
isomorphisms of~$\cat{Sp}$.

This concept matched by that of a ``good''~SGF or~SGNT, based on the
following concept.  Its uniqueness claim will be proven later, to keep
this section efficient.

\begin{theorem}{defn}{alternating}
  We say that an SGF~$F$ is \define{alternating} if it is of the
  form~$(f_*) g^* h_* \cdots$ ($f$~optional), where none of~$f$, $g$,
  $h$,~$\dots$ is the identity map. The \define{alternating reduction}
  of~$F$ is the unique (by \ref{alternating unique}) alternating
  SGF~$F'$ admitting an~SGNT in~$\genby{\comp_0, \triv_0}$ (thus, an
  isomorphism in~$\cSGNT$)~$F \to F'$.
\end{theorem}

\begin{theorem}{defn}{good}
  Let $\cat{Sh} \to \cat{Sp}$ be a geolocalizing geofibered category,
  and let~$F$ be an alternating~SGF; we say that it is \emph{good} if,
  either: every basic~SGF in $F$ of the form~$f^*$ is
  pull-geolocalizing; or, every basic~SGF in $F$ of the form~$f_*$ is
  push-geolocalizing.  For any~SGF, we say that it is good if it has a
  good alternating reduction.  We say that an~SGNT $\phi \colon F \to
  G$ is good if both $F$~and~$G$ are good.
\end{theorem}

\begin{theorem}{lem}{good persistence}
  If~$F$ is good, then for any~SGNT $\phi \colon F \to G$ with $\phi
  \in \comp \cup \triv$, we have~$G$ good as well.  If~$F$ is
  alternating, the same is true for $\phi \in \bc$.
\end{theorem}

\begin{proof}
  Let $\alpha \colon F \to F'$ and $\beta \colon G \to G'$ be the
  alternating reductions of $F$~and~$G$.  If $\phi \in \comp_0 \cup
  \triv_0$, then~$\beta\phi$ is also an alternating reduction of~$F$,
  and thus by \ref{alternating} we have $F' = G'$; since~$F$ (and
  thus~$F'$) is good by hypothesis,~$G'$ (and thus~$G$) is good.
  Likewise, if $\phi \in \comp_0^{-1} \cup \triv_0^{-1}$, then~$\alpha
  \phi^{-1}$ is an alternating reduction of~$G$ and so we have $G' =
  F'$ and thus, again, that~$G$ is good.

  Finally, for $\phi \in \bc$, it is clear first of all that for a map
  $\bc(f,g) \colon g^* f_* \to \tilde{f}_* \tilde{g}^*$ itself, the
  right-hand side is good if the left is, because the hypotheses
  imposed by goodness on the maps of spaces are stable under base
  change.  For a more general element $A \bc(f,g) B$ of~$\bc$, if $F =
  A g^* f_* B$ is good and alternating then $G = A \tilde{f}_*
  \tilde{g}^* B$ satisfies the conditions given in \ref{good} (but is
  not alternating); since the (push-, pull-) geolocalizing morphisms
  are closed under composition and contain the identity, these
  conditions are preserved by passage through any element
  of~$\genby{\comp_0, \triv_0}$, so the same conditions are satisfied
  by the alternating reduction~$G'$; i.e.~$G'$, hence~$G$, is good.
\end{proof}

This definition is complemented and completed by the following
concept, upon which all significant results of this paper are
ultimately based.

\begin{theorem}{defn}{roof}
  Let $\cat{Sh} \to \cat{Sp}$ be a geofibered category and let~$F$ be
  any~SGF, viewed as a diagram of morphisms in~$\cat{Sp}$.  We define
  the \emph{roof} of~$F$, denoted~$\roof(F)$, by constructing the
  final object in the category of spaces with maps to this diagram,
  which exists and is unique by the universal property of the fibered
  product.

  Then the roof is the space thus defined, also denoted~$\roof(F)$,
  together with its ``projection'' morphisms
  \begin{align}
    \label{eq:roof projections}
    a_F \colon \roof(F) \to \target(F) &&
    b_F \colon \roof(F) \to \source(F).
  \end{align}
  % and defined up to unique isomorphism by the three properties
  % \begin{itemize}
  % \item If $f \colon X \to Y$, then $\roof(f_*) = X$ with $b_F = \id$
  %   and $a_F = f$, respectively.
  % \item Likewise, $\roof(f^*) = X$ with $b_F = f$ and $a_F = \id$.
  % \item If $F = GH$ is a composition, we define
  %   \begin{subequations}
  %     \begin{gather}
  %       \label{eq:roof induction}
  %       \roof(F) = \roof(H) \times_{\target(H) = \source(G)} \roof(G) \\
  %       b_F = b_H \on{pr}_{\roof(H)} \qquad
  %       a_F = a_G \on{pr}_{\roof(G)},
  %     \end{gather}
  %   \end{subequations}
  %   where the latter composands are projections onto the factors of a
  %   fibered product.
  % \end{itemize}
  We also define~$\roof(F)$ as an~SGF to be the functor~$a_{F{*}}
  b_F^*$ (resp.~$a_{F{*}}$ if $b_F = \id$, resp.~$b_F^*$ if $a_F =
  \id$, resp.~$\id$),~and in~\ref{roof morphism} we will construct a
  canonical~SGNT $F \to \roof(F)$ that we will also call~$\roof(F)$.
  We will make an effort to eliminate ambiguity.
\end{theorem}

Just as for goodness, the concept of a roof comes associated with a
natural morphism.  We state this proposition here and prove it as
Propositions~\plainref{roof morphism uniqueness} and~\plainref{roof
  morphism construction}.

\begin{theorem}{prop}{roof morphism}
  For any SGF~$F$, its roof is the unique SGF of the form~$a_* b^*$
  admitting admitting a map $\roof(F) \colon F \to \roof(F)$
  in~$\SGNT_0^+$.  This map exists and is uniquely determined by the
  properties that it factors through the alternating reduction of~$F$
  and, if~$F$ is alternating, through an element of~$\bc$.
\end{theorem}

\begin{theorem}{cor}{roof isomorphism}
  Let $\cat{Sh} \to \cat{Sp}$ be a geolocalizing geofibered category.
  If~$F$ is any good~SGF, then~$\roof(F)$ is also good, and $\roof(F)
  \colon F \to \roof(F)$ is a natural isomorphism.
\end{theorem}

\begin{proof}
  In the construction of \ref{roof morphism}, the roof morphism is
  constructed in such a way that its factors in~$\bc$ are applied only
  to an alternating source, with the other factors in~$\comp_0$
  or~$\triv_0$ by \ref{alternating}. Therefore~\ref{good persistence}
  applies and each composand of~$\roof(F)$ is good.  Then each
  base-change factor is, by definition of good, an isomorphism, while
  all the other factors are automatically so.
\end{proof}

This motivates notation for classes of~SGNTs including the inverses of
invertible base changes.

\begin{theorem}{defn}{invertible base changes}
  We use the following class notation for the ``balanced'' elements
  of~$\SGNT$ in which the units and counits only occur in pairs within
  a base change morphism, along with a ``forward'' variant:
  \begin{align}
    \label{eq:SGNT0}
    \SGNT_0^+ = \genby{\bc, \comp_0, \triv_0} &&
    \SGNT_0 = \genby{\SGNT_0^+ \cup (\SGNT_0^+)^{-1}}.
  \end{align}
  Here we use the inverse notation to refer to the class of inverses
  of only the actually invertible (as abstract natural
  transformations) elements of~$\SGNT_0^+$.
\end{theorem}

\subsection*{Invertible unit morphisms.}
We will also be allowing the inverses of certain units and counits,
whose definition is more technical.  The goodness hypothesis in this
next definition is, strictly speaking, unnecessary for its
formulation, but as it is required in our only nontrivial example of
this concept, \ref{etale acyclicity}, it seems likely that without it
the definition would be invalid.

\begin{theorem}{defn}{acyclicity structure}
  Let $\cat{Sh} \to \cat{Sp}$ be a geolocalizing geofibered category.
  We define an \define{acyclicity structure} on it to be a class~$C$
  of pairs~$(a,b)$ of morphisms of~$\cat{Sp}$ having the properties:
  \begin{itemize}
  \item Every pair~$(a, i)$ or~$(i, b)$, with~$a$ being
    push-geolocalizing and~$b$ being pull-geolocalizing, and where $i$
    is any invertible morphism in $\cat{Sp}$ (and having the
    appropriate sources and targets, as below), is in~$C$.
  \item For any~$(a,b) \in C$, we have $\source(a_*) = \target(b^*)$,
    and for every $f \in \cat{Sp}$ such that $a_* \unit(f) b^*$ is
    good and a natural isomorphism, either (\define{left
      invertibility})~$a_* \unit(f)$, or (\define{right
      invertibility})~$\unit(f) b^*$ is a natural isomorphism.
  \item $C$ is closed under base change in the following sense: for
    any pair of maps $X \to \target(a_*)$ and $Y \to \source(b^*)$,
    the base change
    \begin{equation}
      \label{eq:acyclicity base change}
      (\tilde{a}, \tilde{b})
      = X \times_{\target(a_*)} (a,b) \times_{\source(b^*)} Y
    \end{equation}
    of the pair map~$(a,b)$ into $\target(a_*) \times \source(b^*)$ is
    in~$C$.
  \end{itemize}
  This entails the derived concept of admissibility: an~SGF $F$ is
  \define{admissible} if $\roof(F) = (a_F, b_F)$ is in~$C$.  We will
  say, correspondingly, that any $(a,b) \in C$ is itself admissible.
\end{theorem}

We chose ``acyclicity structure'' in reference to a morphism $f \colon
X \to Y$ being acyclic in geometry or topology when $\unit(f)$ is an
isomorphism on certain sheaves (e.g.\ possibly only those of the form
$b^* \sh{F}$), potentially after taking cohomology (i.e.\ applying
derived~$a_*$).

We note that, by definition, every geolocalizing geofibered category
has a ``trivial'' acyclicity structure consisting of just the
pairs~$(a, i)$ and~$(i, b)$ of the first point.

\begin{theorem}{defn}{invertible units}
  Let $\cat{Sh} \to \cat{Sp}$ be a geolocalizing geofibered category
  with acyclicity structure.  We define the class~$\Unit$ to be the
  class of all good~SGNTs of the form~$F \phi G$, where $\phi = A
  \unit(f) B$ is a natural isomorphism and~$\on{dom}(\phi)$ (which
  is~$AB$) is admissible.
\end{theorem}

Until now we have ignored the counits, but for the most part, this is
justifiable.  We give the proof of the following lemma in
\ref{sec:proofs}.

\begin{theorem}{lem}{counits auxiliary}
  We have $\unit' \subset \genby{\SGNT_0^+, \unit}$ and,
  with~$\Unit'$ as in~\ref{invertible units} but with
  \emph{good}~$\unit'(f)$ replacing arbtrary~$\unit(f)$, we have
  $\Unit' \subset \genby{\SGNT_0^+, \Unit}$.
\end{theorem}

\section{Main theorems}
\label{sec:main thms}

In this section we suppose the existence of an ambient geolocalizing
geofibered category with an acyclicity structure, $\cat{Sh} \to
\cat{Sp}$; as we have noted, any geofibered category can play this
role with trivial structures; among our results is a description of
some nontrivial ones.  Our main results are of two types: the first
contains a ``quantitative'' and comparatively technical result on
SGNTs; the second contains a ``qualitative'' corollary.  Proofs, if
not indicated otherwise, are given in \ref{sec:proofs}.

The quantitative result is a classification of~SGNTs:

\begin{theorem}{thm}{SGNT0 inverses square}
  Let $\phi \colon F \to G$ be in~$\genby{\SGNT,\bc^{-1},\Unit^{-1}}$,
  and denote $\roof(G) = a_{G{*}} b_G^*$.  Then there exist maps of
  spaces $f$~and~$g$, such that~$a_{G{*}} \unit(g) b_G^*$ is a natural
  isomorphism, forming a commutative diagram:
  \begin{equation}
    \label{eq:SGNT0 inverses square}
    \includestandalone{img37}
  \end{equation}
  The upward arrow can be omitted for~$\phi \in \genby{\SGNT,\bc^{-1}}$.
\end{theorem}

The last sentence is \ref{SGNT0 square}, and the rest is \ref{SGNT0
  inverses square'}.  The use of~$\genby{\SGNT \cup \bc^{-1}}$ rather
than~$\genby{\SGNT_0, \unit}$ is justified by \ref{counits auxiliary}.

Now we give criteria under which ``all diagrams commute''; i.e.\ there
exists only one~SGNT between two given SGFs. First, we have the basic
Corollaries~\plainref{big alternating unique} and~\plainref{SGNT0
  roofs}:

\begin{theorem}{thm}{basic theorem}
  \mbox{}
  \begin{enumerate}
  \item For any~SGFs $F$~and~$G$, there exists at most one $\phi
    \colon F \to G$ in~$\genby{\comp, \triv}$.

  \item Let~$G = a_* b^*$; then for any SGF~$F$, there exists at most
    one $\phi \colon F \to G$ in~$\SGNT_0$.
  \end{enumerate}
\end{theorem}

We have also found a curious conclusion that is not quite a corollary
of the latter nor of the main theorem:

\begin{theorem}{thm}{secondary theorem}
  Let $\phi \colon F \to G$ be in~$\SGNT$, where both $F$~and~$G$ have
  at most one basic~SGF.  Then either both are trivial and~$\phi \in
  \triv_0^{-1} \triv_0$, or neither is, so $F = G$ and $\phi = \id$.
\end{theorem}

More importantly, we have the following general theorem:

\begin{theorem}{thm}{main theorem}
  Let~$\phi \colon F \to G$ be a natural transformation of~SGFs.
  
  Write $\source(F) = \source(G) = X$ and $\target(F) = \target(G) =
  Y$; denote $Z = X \times Y$ and let $b \colon \roof(F) \times_Z
  \roof(G) \to \roof(G)$ be the projection map, where $\roof(F) \to Z$
  and $\roof(G) \to Z$ are the pair maps $(b_F, a_F)$~and~$(b_G,
  a_G)$; suppose that the unit map $a_{G{*}} \unit(b) b_G^*$ is an
  isomorphism.

  Suppose as well that the map $\roof(G) \colon G \to \roof(G)$ is an
  isomorphism. If $\phi \in \genby{\SGNT, \bc^{-1}, \Unit^{-1}}$, then
  it is the unique map $\psi \colon F \to G$ in that class.
\end{theorem}

We also have an auxiliary lemma giving sufficient conditions for the
hypotheses of the theorem to hold.  To state it, we use the term
\define{weakly admissible} of an SGF~$F$ to mean that the pair
morphism~$(a_F, b_F)$ of its~roof is a universal monomorphism.

\begin{theorem}{lem}{conditions auxiliary}
  We have $\roof(G)$ is an isomorphism if~$G$ is good.  The other
  condition of \ref{main theorem} holds if either:
  \begin{itemize}
  \item $F$ is weakly admissible, and either there exists some $\phi
    \in \genby{\SGNT_0, \unit}$, or its consequence:~$(b_G, a_G)$
    factors through~$(b_F, a_F)$;
  \item $G$ is weakly admissible and there exists some $\phi \in
    \genby{\SGNT_0, \Unit^{-1}}$.
  \end{itemize}
\end{theorem}

Finally, we address the question of exhibiting geolocalizing and
acyclicity structures on the geofibered categories that occur in
practice.  The intended application of this concept is to ``sheaves''
on ``schemes'' in various more or less geometric contexts:
\begin{description}
\item[Presheaves] We may define $\cat{Sp} = \cat{Cat}^\op$, the
  opposite category of all small categories, and for any category~$X$,
  set $\cat{Sh}_X$ to be the category of presheaves on~$X$ (with
  values in any fixed complete category; if it is abelian, then so is
  $\cat{Sh}_X$ for any $X$).  For any morphism $f \colon X \to Y$ of
  spaces (i.e.\ a functor $F \colon Y \to X$) and for any shapes
  (presheaves) $\sh{F} \in \cat{Sh}_X$ and $\sh{G} \in \cat{Sh}_Y$, we
  take the usual pushforward and pullbacks:
  \begin{align}
    \label{eq:presheaf ops}
    f_*(\sh{F})(y) = \sh{F}(F(y)) &&
    f^*(\sh{G})(x) = \lim_{x \to F(y)} \sh{F}(y)
  \end{align}
  This is far more general than actual sheaves on schemes, but because
  of the flexibility in the concept of geolocalizing and acyclicity
  structures, our results apply uniformly to it (if with potentially
  less power should these structures be too trivial).
\item[Quasicoherent sheaves] Each $\cat{Sh}_X$ is the category of
  quasicoherent sheaves in the Zariski topology on $X \in \cat{Sp}$
  being a scheme of finite type over a fixed locally noetherian base
  scheme; pushforwards and pullbacks are those of quasicoherent
  sheaves.
\item[Constructible \'etale sheaves] Each $\cat{Sh}_X$ is the category
  of $\ell$-torsion or $\ell$-adic ``sheaves'' in the \'etale topology
  on $X$ being a scheme of finite type over a fixed locally noetherian
  base scheme; pushforwards and pullbacks are those of such
  ``sheaves'' (ultimately, inherited from actual sheaf operations) We
  will not recall the definition here, nor the definitions of any of
  the functors on it, but work purely with the associated formalism.
\item[Constructible complex sheaves] Each $\cat{Sh}_X$ is the category
  of constructible sheaves of complex vector spaces in the classical
  topology on $X$ being a complex-analytic variety of finite type over
  a fixed base variety.  Pushforwards and pullbacks are those of
  sheaves of complex vector spaces.
\item[Derived categories] With $\cat{Sp}$ being any of the above
  categories of spaces, we may take $\cat{Sh}_X$ to be the derived
  category of the corresponding abelian category of shapes on a
  space~$X$.  Pushforwards and pullbacks are, respectively, the
  right-derived pushforward and left-derived pullback.
\end{description}
The noetherian hypotheses were suggested by Brian Conrad to ensure the
good behavior of the sheaf theory, as we have attempted to encapsulate
in \ref{geofibered category} and subsequent definitions.  Presumably
this list, as varied as it is, is incomplete; for instance, most
likely sheaves on the crystalline site, D-modules, and other such
categories belong on it as well.

We regret that we have been unsuccessful in locating references that
explicitly prove, in all of these contexts, that the functors
described (which \emph{are} defined very carefully) actually have the
properties that we have called a geofibered category.  In all cases,
the pseudofunctor structure of pushforwards is either totally obvious
(as for presheaves and sheaves, given~\ref{eq:presheaf ops}) or formal
(as for $\ell$-adic and derived sheaves), and in lieu of existing
literature on the category-theoretic minutiae, we feel free to simply
declare that the structure for pullbacks should be \emph{determined}
by adjunction and the required compatibilities; see the discussion
following \ref{geofibered category}.

In order to exhibit geolocalizing structures on these categories, we
simply recall the base change theorems of algebraic geometry, together
with standard properties of the types of morphism.

\begin{theorem}{lem}{base change theorems}
  (\textit{Proper, smooth,~and flat base change}) In the \'etale or
  complex (possibly derived) contexts, the class of proper morphisms
  of schemes is push-geolocalizing and the class of smooth morphisms
  is pull-geolocalizing. In the quasicoherent (possibly derived)
  context, the class of flat morphisms is pull-geolocalizing. \qed
\end{theorem}

A stunningly general flat base change theorem for derived
quasicoherent sheaves on any algebraic spaces over any scheme can be
found at~\cite{stacks-project}*{Tag 08IR}.  The proper and smooth base change
theorems in \'etale cohomology were proven for torsion sheaves by
Artin~\cite{sga4}*{Exp.~xii, xiii, xvi}; the $\ell$-adic and derived
versions follow formally.  The recent preprint~\cite{enhanced} of Liu
and Zhang presents an extension of these theorems to the derived
categories on Artin stacks in the lisse-\'etale topology, as well.

As for acyclicity structures, in general we can only offer the trivial
one, but in the \'etale context or its derived analogue (and
presumably by the same token, the complex one) we can do better using
a theorem from SGA4.

\begin{theorem}{lem}{etale acyclicity}
  In the \'etale or derived \'etale contexts, the class
  \begin{equation}
    \label{eq:etale acyclicity}
    C = \{(a,b) \colon X \to Y \times Z \mid
          \text{$(a,b)$ is an immersion}\}
  \end{equation}
  is an acyclicity structure for the geolocalizing structure defined
  in \ref{base change theorems}.
\end{theorem}

\section{String diagrams}
\label{s:string diagrams}

In the course of executing the general strategy of~\ref{sec:simplify}
we will need to do a few specific computations with~SGNTs.  As these
have little intrinsic meaning, the work would be unintelligible using
traditional notation, so we have chosen to express it visually using
``string diagrams''.

For the convenience of readers familiar with such depictions of
categorical algebra, in the present section we will give only the
essential definitions and results that will be cited in our later
proofs.  A more conversational introduction to ths topic of string
diagrams, together with the proofs of the mostly routine facts shown
here, are left to the appendix.

In summary, in our diagrams, vertical edges represent basic geometric
functors, and are marked with upward or downward arrows to
distinguish, respectively, $f_*$ from $f^*$.  The shapes shown in
\ref{fig:basic SGNTs} generate all our string diagrams by horizontal
(natural transformation) and vertical (functor) composition, which
correspond to horizontal (left-to-right) and vertical (bottom-to-top)
concatenation of diagrams.  We use a doubled-line convention for our
edges, which has no mathematical meaning but does improve aesthetics
and (with some imagination) topologically justifies most of our string
diagram identities as being mere topological deformations in the
plane.

\begin{figure}[p]
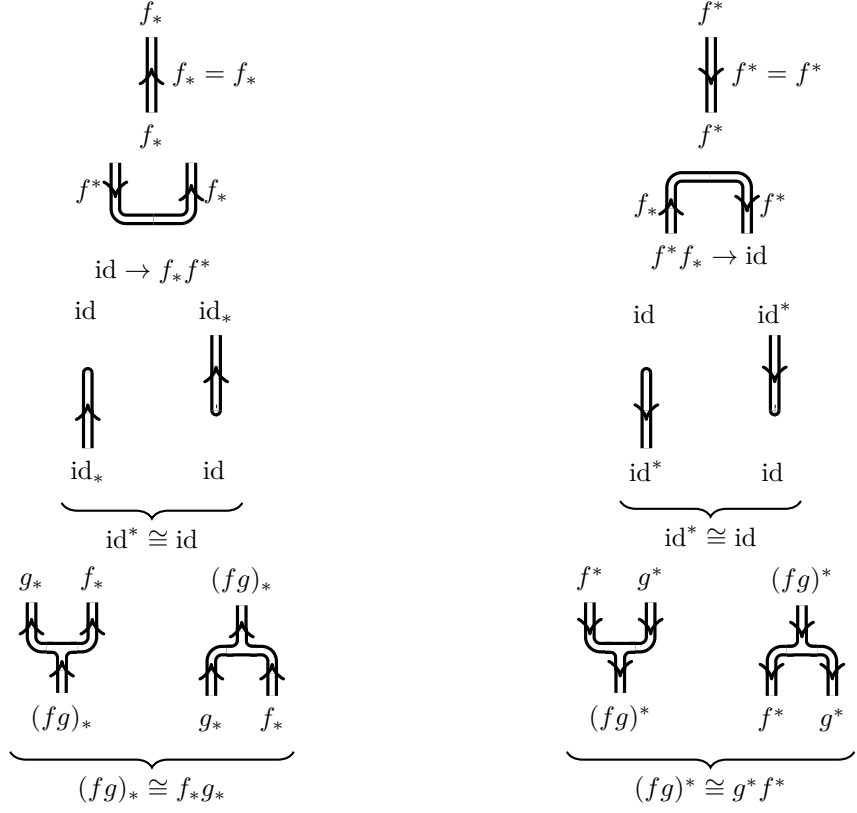

  \begin{tabu}{X[-1,c]X[-1,c]}
    % Identity maps
    \adjustbox{lap={\widthof{$f_* = f_*$}}}{\includestandalone{img1a}} &
    \adjustbox{lap={\widthof{$f^* = f^*$}}}{\includestandalone{img1b}} \\
    % U-shapes
    \includestandalone{img1c} & \includestandalone{img1d} \\
    % i-shapes
    \includestandalone{img1e} & \includestandalone{img1f} \\
    % Y-shapes
    \includestandalone{img1g} & \includestandalone{img1h}
  \end{tabu}
  \caption{The basic SGNTs as string diagrams}
  \label{fig:basic SGNTs}
\end{figure}
\begin{figure}[p]
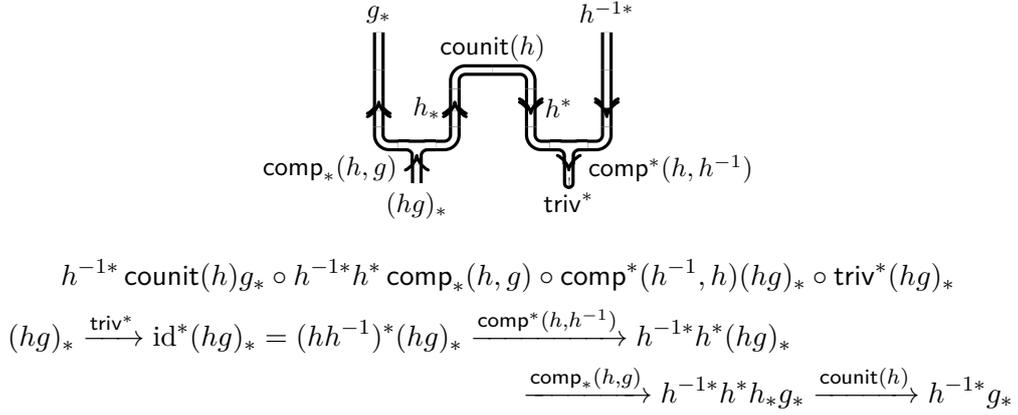

  \centering

  \includestandalone{img2}
  \begin{gather*}
    h^{-1{*}}\unit'(h)g_* \circ
    h^{-1{*}}h^*\comp_*(h,g) \circ
    \comp^*(h^{-1},h) (hg)_* \circ
    \triv^* (hg)_*
    \\
    \begin{multlined}[0.9\linewidth]
      (hg)_*
      \xrightarrow{\triv^*}
      \id^* (hg)_* = (h h^{-1})^* (hg)_*
      \xrightarrow{\comp^*(h,h^{-1})}
      h^{-1{*}}h^* (hg)_* \\
      \xrightarrow{\comp_*(h,g)}
      h^{-1{*}}h^*h_* g_*
      \xrightarrow{\unit'(h)}
      h^{-1{*}} g_*
    \end{multlined}
  \end{gather*}
  \caption{Example of a string diagram with corresponding SGNT
    expressed several ways}
  \label{fig:diagram}
\end{figure}

An example of the correspondence between string diagrams and SGNTs is
given in \ref{fig:diagram}, but we will never be so thorough in
labeling them in actual use.

We emphasize that it is important for the correspondence between
string diagrams and natural transformations that the string diagram be
\emph{labeled}; i.e.~for the edges and components to have the meanings
assigned to them by~\ref{fig:basic SGNTs}.  For if not, then the
operation of reflecting the string diagram vertically produces another
planar graph that is valid combinatorially, but does not necessarily
correspond to any~SGNT (the corresponding operation on functors $f^*
g_* h^* \cdots$ is to swap upper~$^*$ and lower~$_*$, but the
resulting basic~SGFs are no longer composable).  We thank Mitya
Boyarchenko for this last observation, however much it forced the
restructuring of this paper.

Although the diagram acquires its unique identity as an~SGNT only
after labeling all the edges, we will almost always omit these labels.
We will never assert the identity of an~SGNT corresponding to an
unlabled diagram without indicating how we would label it, which can
(in our applications) always be deduced from the ends by propagating
through the various transformations.

\subsection*{String diagram identities.}
In this subsection we record all the identities satisfied by string
diagrams with two shapes.  These can basically be considered the
relations in the category whose objects are~SGFs and whose morphisms
are the string diagrams between two given~SGFs.  The proofs, which are
either direct translation of the corresponding symbolic equations or
simple manipulation, are left to the appendix for the convenience of
readers who would prefer to get to the point.

\begin{theorem}{lem}{fig:adjunctions}
  ({Adjunction identities})

  \nopagebreak\medskip
  \centering

  \includestandalone{img3a} \hfil \includestandalone{img3b}
\end{theorem}

\begin{theorem}{lem}{fig:compositions (inv)}
  ({Composition identities (inverses)})

  \nopagebreak\medskip
  \centering

  \includestandalone{img4a} \hfil \includestandalone{img4b}

  \includestandalone{img4c} \hfil \includestandalone{img4d}
\end{theorem}

\begin{theorem}{lem}{fig:compositions (assoc)}
  ({Composition identities (associativity)})

  \nopagebreak\medskip
  \centering

  \includestandalone{img5a} \hfil \includestandalone{img5b}
  
  \includestandalone{img5c} \hfil \includestandalone{img5d}
\end{theorem}

\begin{theorem}{lem}{fig:trivializations triv}
  ({Trivialization identities (trivializations)})

  \nopagebreak\medskip
  \centering

  \includestandalone{img6a} \hfil \includestandalone{img6b} \hfil
  \includestandalone{img6c} \hfil \includestandalone{img6d}
\end{theorem}

\begin{theorem}{lem}{fig:trivializations adj}
  ({Trivialization identities (adjunctions)})

  \nopagebreak\medskip
  \centering

  \includestandalone{img7a} \hfil \includestandalone{img7b} \hfil
  \includestandalone{img7c} \hfil \includestandalone{img7d}
\end{theorem}

\begin{theorem}{lem}{fig:trivializations comp}
  ({Trivialization identities (compositions)})

  \nopagebreak\medskip
  \centering

  \includestandalone{img8a}

  \includestandalone{img8b}
\end{theorem}

\begin{theorem}{lem}{fig:adjunction-composition 1}
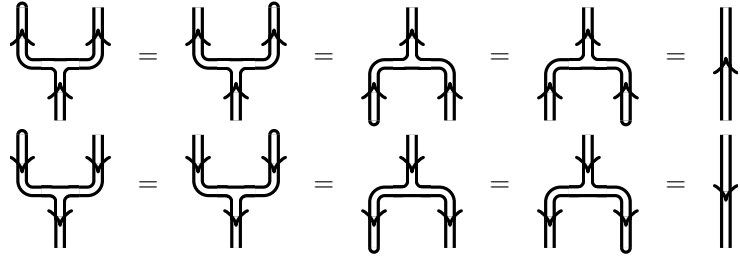
  ({Adjunction-composition identities (part 1)})

  \nopagebreak\medskip
  \centering
  % Unit-composition (handle version)
  \includestandalone{img9a} \hfil \includestandalone{img9b}
  
  \medskip
  \includestandalone{img9c} \hfil \includestandalone{img9d}
\end{theorem}

\begin{theorem}{lem}{fig:adjunction-composition 2}
  ({Adjunction-composition identities (part 2)})

  \nopagebreak\medskip
  \centering
  % Unit-composition (fork version)
  \includestandalone{img10a} \hfil \includestandalone{img10b}
  
  \medskip
  \includestandalone{img10c} \hfil \includestandalone{img10d}
\end{theorem}

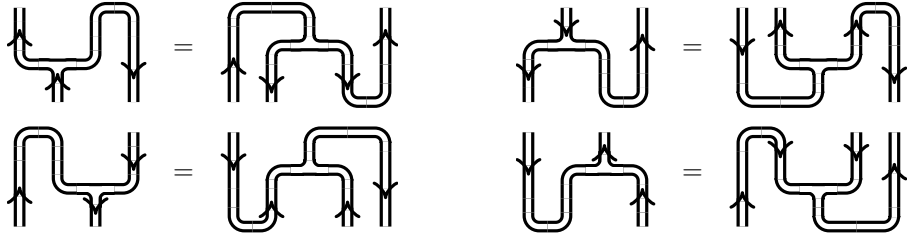
\begin{theorem}{lem}{fig:compositions (nontrivial)}
  Suppose we have maps of spaces~$f$, $g$, $h$,~and $k$ and another
  one~$l$ such that $g = lk$~and $h = fl$ (resp.~$f = hl$~and $k =
  lg$).  Then we have the first equality (resp.~the second equality)
  below, and similarly for the~$^*$ version:
  \begin{equation}
    \label{eq:compositions (nontrivial) diagram}
    \includestandalone{img11}
  \end{equation}
\end{theorem}

We note that \cite{mccurdy}*{Definition~4} establishes a
\define{Frobenius algebra} as an object in a monoidal category
satisfying precisely the above diagrammatic constraints, except that
here, the ends of that diagram are not all the same object.  This
lemma therefore shows that the basic~SGFs of a geofibered category
form a generalization of a Frobenius algebra, presumably a ``Frobenius
algebroid'' in the same sense as a groupid, though we have not been
able to find this term in use.

\begin{theorem}{lem}{SD loop}
  We have the equivalence (for either direction of edges and all
  possible assignments of maps of spaces compatible with the depicted
  compositions):
  \begin{equation}
    \label{eq:SD loop}
    \includestandalone{img12}
  \end{equation} 
\end{theorem}

\section{Uniqueness of reduced forms}
\label{sec:uniqueness}

We have defined (Definitions~\plainref{alternating}
and~\plainref{roof}) two reduced forms for an~SGF and claimed that
they are unique.  In this section we prove these claims; the first one
is relatively straightforward and does not require string diagrams,
but the manipulations of the second one are simplified by their use,
so we have placed both of them at this point.  Only one general
relation needs to be noted here, expressing the ``commutation'' of
unrelated natural transformations.

\begin{theorem}{lem}{horizontal composition verify}
  Let $F = A F_1 B F_2 C$ be a composition of functors; let $\phi
  \colon F_1 \to G_1$ and $\psi \colon F_2 \to H_2$ be natural
  transformations.  Then the following diagram commutes:
  \begin{equation}
    \label{eq:horizontal composition verify}
    \includestandalone{img13} \qquad \qed
  \end{equation}
\end{theorem}

\subsection*{Uniqueness of the alternating reduction.}

Showing that the alternating reduction is unique is a matter of
enforcing an inductive structure on its construction, as encapsulated
by the following lengthy definition.

\begin{theorem}{defn}{staged morphism}
  Let $\phi \colon F \to G$ be an~SGNT; we define a \define{staging
    structure} on it to be the following data:
  \begin{itemize}
  \item A sequence of~SGNTs
    \begin{equation}
      \label{eq:stages}
      F = F_0
      \xrightarrow{\gamma_1} F_1
      \xrightarrow{\alpha_2} F_2
      \cdots
      \xrightarrow{\alpha_{2n}} F_{2n}
      = G
    \end{equation}
    together with representations of $\gamma_i \in \genby{\comp_0}$
    and $\alpha_i \in \genby{\triv_0}$ as products of \define{factors}
    that are basic~SGNTs.

  \item This data must satisfy some conditions.  We make reference to
    the \define{terms} of an~SGF, its basic~SGF composands; a term is
    \define{trivial} if it is of the form $\id_*$~or~$\id^*$; a
    \define{composable pair} is a sequence of two consecutive terms of
    the form $f_* g_*$~or~$f^* g^*$.  We define the \define{stage} of
    any trivial term or composable pair in any of the
    intermediate~SGFs of the staging structure, together with
    conditions:
    \begin{itemize}
    \item The stage of any composable pair of~$F_0$ is~$0$, and of any
      trivial term is~$1$.
    \item Let~$\alpha$ be any factor of~$\alpha_i$, acting locally
      as~$t_1 t t_2 \to t_1 t_2$, where~$t$ is trivial; we require
      that~$t$ have stage~$i - 1$ and that that both pairs $t_1
      t$~and~$t t_2$ (when composable) have stage~$i$.  If~$t_1 t_2$
      is a composable pair, we define its stage to be~$i$.  We say
      that~$t$, $t_1 t_2$,~and $t_1 t$~and~$t t_2$ are
      \define{affected by}~$\alpha$.
    \item Let~$\gamma$ be any factor of~$\gamma_i$, acting locally as
      $t_0 t_1 t_2 t_3 \to t_0 t t_3$, where~$t_1 t_2$ is composable.
      We require that~$t_1 t_2$ have stage~$i -1$, and that if either
      of these terms is trivial, then that term has stage~$i$.  We
      define the stages of $t_0 t$~or~$t t_1$ (when composable), to be
      those of $t_0 t_1$~and~$t_1 t_2$ each; the stage of~$t$ (if it
      is trivial) is~$i$, unless $t_1$~and~$t_2$ are both trivial, in
      which case the stage of~$t$ is the larger of their stages.  We
      say that~$t$ and~$t_1 t_2$ are \define{affected by}~$\gamma$.
    \end{itemize}
  \end{itemize}
  We say that the staging structure is \define{complete} at stage~$s$
  when~$s$ is odd if~$F_s$ has no composable pairs, and when~$s$ is
  even if~$F_s$ has no trivial terms.
\end{theorem}

The following fact is trivial:

\begin{theorem}{lem}{stage limit}
  In a staged~SGNT $\phi \colon F \to G$, any trivial term or
  composable pair in~$F_s$ has stage at most~$s$, if~$s > 0$. A term
  of stage~$s$ in~$G$ is not affected by any factor of any
  $\alpha_i$~or~$\gamma_i$ with~$i > s$. \qed
\end{theorem}

In the next few lemmas we establish the strong properties of a
complete staging.

\begin{theorem}{lem}{complete staging}
  If an~SGNT with staging $\phi \colon F \to G$ is complete at any
  stage~$s \geq 3$, then it is also complete at stage~$s - 2$.
\end{theorem}

\begin{proof}
  Suppose~$s$ is odd, and consider a composable pair~$t_1 t_2$
  in~$F_{s - 2}$; it has stage~$\leq s - 2$ by \ref{stage limit}.  It
  is thus not possible for either term to be affected by a factor
  of~$\alpha_{s - 1}$, so this pair persists into~$F_{s - 1}$.  There,
  since~$t_1 t_2$ still has stage~$\leq s - 2$, it cannot be acted on
  by a factor of~$\gamma_s$ so this space in~$F_{s - 1}$ will remain
  in a composable pair in~$F_s$.
  
  Suppose~$s$ is even, and consider a trivial term~$t$ in~$F_{s - 2}$.
  By \ref*{stage limit} it has stage~$\leq s - 2$, so it is not
  affected by any factor of~$\gamma_{s - 1}$, so it persists
  into~$F_{s - 1}$.  There, it has the same stage and so cannot be
  affected by any factor of~$\alpha_s$, so it persists into~$F_s$,
  forming a trivial term there.
\end{proof}

\begin{theorem}{cor}{complete staging induct}
  If $\phi \colon F \to G$ has a staging such that~$G$ has neither
  composable pairs nor trivial terms, then it is complete at every
  stage~$s \geq 1$.
\end{theorem}

\begin{proof}
  It is vacuously compatible with the definition of a staging to
  define $F_{2n + 2} = F_{2n + 1} = F_{2n} = G$ with $\alpha_{2n + 2}
  = \gamma_{2n + 1} = \id$, with which convention~$\phi$ is complete
  at stages $2n + 1$~and~$2n$, so by induction on \ref{complete
    staging}, at every stage~$s \geq 1$.
\end{proof}

\begin{theorem}{lem}{complete staging determined}
  If $\phi \colon F \to G$ has a staging that is complete at every
  stage~$s \geq 1$, then all the intermediate SGFs~$F_i$ and~SGNTs
  $\alpha_i$~and~$\beta_i$ are uniquely determined by~$F$.
\end{theorem}

\begin{proof}
  It is clear that if~$\phi$ is complete at~$F_1$, then~$F_1$ must be
  obtained by composing all composable pairs of~$F_0$; this is
  well-defined regardless of order by~\ref{eq:compositions assoc} and
  gives~$\gamma_1$ uniquely.  Then it is also clear that if~$\phi$ is
  complete at~$F_2$, it must be obtained by deleting all trivial terms
  of~$F_1$, which is well-defined regardless of order
  by~\ref{horizontal composition verify} and gives~$\alpha_2$
  uniquely.  By induction, the lemma follows.
\end{proof}

Now we turn to the construction of a staging on a given~SGNT.

\begin{theorem}{lem}{stage associativity}
  Let $\phi \colon F \to G$ have a staging; then any reordering of the
  factors of any~$\gamma_i$ by associativity, as
  in~\ref{eq:compositions assoc}, is a valid staging.
\end{theorem}

\begin{proof}
  Given two overlapping compositions:
  \begin{align}
    \label{eq:stage associativity}
    t_0 ((t_1 t_2) t_3) t_4
    &&
    t_0 (t_1(t_2 t_3)) t_4
  \end{align}
  suppose that the first grouping satisfies the conditions of a
  staging.  Thus,~$t_1 t_2$ has stage~$\geq i - 1$ and each
  trivial~$t_j$ has stage~$i$.  Let~$t$ be their composition, so in
  particular,~$t t_3$ has the same stage as~$t_2 t_3$ and, as it is
  acted on by another factor of~$\gamma_i$, thus has stage~$\geq i -
  1$.  It follows that~$t_2 t_3$ has stage~$\geq i - 1$, and likewise
  that if~$t_3$ is trivial, then it has stage~$i$.  Therefore the
  second grouping also satisfies the conditions of a staging.
  Furthermore, the product of the first group has stage~$i$ unless all
  three of~$t_1$, $t_2$, and~$t_3$ are trivial, which is the same
  exception as for the second group, in which case its stage is the
  maximum of their stages.  So the results of each grouping are
  identical.
\end{proof}

\begin{theorem}{lem}{stage comp}
  Let $\phi \colon F \to G$ have a staging and let~$t_0 t_1 t_2 t_3
  t_4 t_5$ be a sequence of composable terms in~$G$.  Suppose that the
  staging can be extended by composing~$t_2 t_3$ \emph{or} by
  composing both $t_1 t_2$~and~$t_3 t_4$.  Then it can be extended by
  the sequence of compositions
  \begin{equation}
    \label{eq:stage comp}
    t_0 (t_1 ((t_2 t_3) t_4)) t_5.
  \end{equation}
\end{theorem}

\begin{proof}
  We suppose that we are constructing stage~$2n + 1$ of~$\phi$.  By
  the first assumption, the stage of~$t_2 t_3$ is~$2n$.  By
  construction, $t = (t_2 t_3)$, if trivial, has stage~$2n + 1$ and
  both pairs~$t_1 t$ and~$t t_4$ have the same stages as,
  respectively, $t_1 t_2$~and~$t_3 t_4$. By the second assumption, the
  latter stages are~$2n$ and each~$t_{1,4}$ (when trivial) has
  stage~$2n + 1$.  Therefore, the pair~$t t_4$ may be composed in a
  staging, with value~$u$ (if trivial) of stage~$2n + 1$ and the
  pair~$t_1 u$ having stage that of~$t_1 t$, which is the same as that
  of~$t_1 t_2$, which is~$2n$.  So the pair~$t_1 u$ may be composed as
  well.
\end{proof}

\begin{theorem}{cor}{stage extend comp}
  Let $\phi \colon F \to G$ have a staging ending at stage~$2n + 1$,
  and let~$\psi \in \comp_0$ have domain~$G$ affecting a pair~$t_1
  t_2$ of stage~$\leq 2n - 2$.  Define a factorization of~$\gamma_{2n
    + 1}$, using \ref{stage limit} and
  associativity~\ref{eq:compositions assoc}, as~$\gamma^1\gamma^0$,
  where~$\gamma^0$ does not affect~$t_1 t_2$ and~$\gamma^1$ is a pair
  of two~$\comp_0$ factors having values $t_1$~and~$t_2$ respectively.
  Let~$\phi'$ be the portion of~$\phi$ up to stage~$2n$ and
  let~$\psi'$ be~$\psi$ applied to the codomain of~$\gamma^0$.  Then,
  if~$\psi' \gamma^0 \phi'$ has a staging, so does~$\psi \phi$.
\end{theorem}

\begin{proof}
  By \ref{stage associativity}, the SGNT~$\gamma^1 \gamma^0 \phi'$ has a
  staging, so by definition, so does~$\gamma^0 \phi'$, and then
  \ref{stage comp} applies.
\end{proof}

\begin{theorem}{lem}{staging induct}
  Let $\phi \colon F \to G$ have a staging, and let~$\psi$ be a
  basic~SGNT either in $\comp_0$~or~$\triv_0$ composable with~$\phi$.
  Then~$\psi\phi$ has a staging.
\end{theorem}

\begin{proof}
  The proof is by descending induction, with the two terminal cases:
  \begin{enumerate}
  \item\label{en:staging triv} $\psi \in \triv_0$ acts as~$t_1 t t_2
    \to t_1 t_2$, where the triple satisfies the conditions of a
    staging at stage~$2n$;
  \item\label{en:staging comp} $\psi \in \comp_0$ acts as~$t_1 t_2 \to
    t$, where the composable pair satisfies the conditions of a
    staging at stage~$2n + 1$;
  \end{enumerate}
  In the first, we can append~$\psi$ to~$\alpha_{2n}$, satisfying the
  definition of a staging.  In the second, we can begin~$\gamma_{2n +
    1}$ with~$\psi$, satisfying the definition of a staging.  Note
  that these each apply, respectively, at stages $2$~and~$1$, so the
  induction is well-founded.

  Suppose that~$\psi \in \triv_0$ acts on a trivial term~$t$;
  if~\ref*{en:staging triv} does not apply, then either~$t$ has
  stage~$<2n - 1$ or one of the pairs including~$t$ is composable; if
  the conditions for a composition in a staging do not apply to it,
  then either the pair has stage~$<2n$ or the other term is trivial of
  stage~$<2n - 1$.

  Suppose that~$\psi \in \comp_0$ acts on a composable pair~$t_1 t_2$;
  if~\ref*{en:staging comp} does not apply, then either its stage
  is~$<2n$, or one of the terms~$t_i$ is trivial of stage~$<s$.

  In either case, we have identified a pair~$t_1 t_2$ such that either
  at least one term is trivial and has stage~$<2n - 1$, or~$t_1 t_2$
  itself has stage~$<2n$; when~$\psi \in \triv_0$ it acts on one of
  the~$t_i$, and when~$\psi \in \comp_0$ it acts on~$t_1 t_2$.  Then
  by~\ref{eq:trivializations comp}, we can replace~$\psi$ with either
  the composition of~$t_1 t_2$ or a trivialization of either
  trivial~$t_i$.  In the latter case, by \ref{horizontal composition
    verify}, we may ``commute''~$\psi$ with all
  $\beta_i$~and~$\gamma_i$ down to~$F_s$.  In the former case, we do
  this with~$\psi \in \comp_0$ using \ref{stage extend comp} and
  \ref*{horizontal composition verify}.  Either way, the proof follows
  by induction.
\end{proof}

\begin{theorem}{prop}{alternating unique}
  For any SGF~$F$, there exists a unique map $F \to G$
  in~$\genby{\comp_0, \triv_0}$ with~$G$ alternating.
\end{theorem}

\begin{proof}
  Any such~SGNT has a staging by induction on \ref{staging induct},
  and~$G$ has neither trivial terms nor composable pairs by definition
  of an alternating~SGF.  Therefore, by \ref{complete staging induct},
  it is complete at every stage, so by \ref{complete staging
    determined} it is uniquely determined by~$F$.  Given~$F$ alone,
  such a map exists by the construction in the proof of that lemma.
\end{proof}

\begin{theorem}{cor}{big alternating unique}
  If $\phi \colon F \to G$ is in~$\genby{\comp, \triv}$, then it is
  unique there, and $F$~and~$G$ have the same alternating reduction.
\end{theorem}

\begin{proof}
  Write~$\phi$ as an alternating composition of~$\genby{\comp_0,
    \triv_0}$ and its inverses.  By \ref{alternating unique}, either
  one preserves both the alternating reduction and the map to it, so
  this is true of~$\phi$ as a whole by induction.  If there is another
  such map~$\phi'$, then~$\phi^{-1} \phi'$ is a self-map of~$F$
  preserving the map to its alternating reduction.  Since that map is
  an isomorphism, we have~$\phi^{-1} \phi' = \id$.
\end{proof}

\subsection*{Relations.}
To deal with the roof and its additional complications arising from
the base change morphisms, we prove a number of~``commutation
relations'' among~$\comp_0$, $\triv_0$,~and $\bc$, beginning with
rewriting some especially trivial transformations in terms of simpler
ones.  In this subsection, we use~$\cd$ instead of~$\bc$ to emphasize
the fundamentally diagrammatic nature of the arguments; following
that, we will be forced for practical reasons to specialize.

\begin{theorem}{lem}{CD trivial isomorphisms}
  We have identities
  \begin{subequations}
    \label{eq:CD trivial isomorphisms}
    \begin{alignat}{99}
      \label{eq:CD trivial isomorphisms lower}
      \cd(f,\id; f, \id) &= &\triv_*^{-1} f^* &\circ f^* \triv_*
      &&\colon f^* \id_* &\to \id_* f^* \\
      \label{eq:CD trivial isomorphisms upper}
      \cd(\id,f; \id, f) &= &f_* \triv^{{*}{-1}} &\circ \triv^* f_*
      &&\colon \id^* f_* &\to f_* \id^*.
    \end{alignat}
  \end{subequations}
\end{theorem}

\begin{proof}
  First, we note that according to the definition~\ref{eq:base change
    definition}, we have as a string diagram
  \begin{equation}
    \label{eq:CD string diagram}
    \cd(f,g; \tilde{f}, \tilde{g}) = \includestandalone{img14}
  \end{equation}
  Now, for the claimed identities, moving the right-hand side terms to
  the left, they are equivalent to the two string diagram identities
  \begin{align}
    \label{eq:CD trivial string diagrams}
    \includestandalone{img15a} && \includestandalone{img15b}
  \end{align}
  the first of which follows from~\ref{fig:trivializations adj} and
  then~\ref{fig:trivializations comp} (twice each),~and the second of
  which follows from \ref*{fig:trivializations comp} (twice) and
  then~\ref{fig:adjunctions}.
\end{proof}

For those nontrivial~SGNTs that do~``interact'', we have the following
relations:

\begin{theorem}{lem}{commutation relations}
  The following~``commutation relations'' hold in~$\SGNT_0$:
  \begin{enumerate}[label=\alph*.,ref=\alph*]
  \allowdisplaybreaks
  \item \label{en:comp-triv commutation}
    We have, for any composable maps~$f$ and~$g$ of spaces:
    \begin{subequations}
      \label{eq:comp-triv commutation}
      \begin{gather}
        \label{eq:comp-triv commutation lower lower}
        \begin{multlined}[b][\linewidth - 6em]
          \bigl(
          f_* \id_* g_*
          \xrightarrow{\triv_*}
          f_* g_*
          \xrightarrow{\comp_*(g,f)}
          (fg)_*
          \bigr) \\
          = \bigl(
          f_* \id_* g_*
          \xrightarrow{\comp_*(g,\id)}
          f_* g_*
          \xrightarrow{\comp_*(g,f)}
          (fg)_*
          \bigr) 
        \end{multlined}
        \\
        \label{eq:comp-triv commutation lower upper}
        \begin{multlined}[b][\linewidth - 6em]
          \bigl(
          f_* \id^* g_*
          \xrightarrow{\triv^*}
          f_* g_*
          \xrightarrow{\comp_*(g,f)}
          (fg)_*
          \bigr) \\
          = \bigl(
          f_* \id^* g_*
          \xrightarrow{\cd(g,\id;g,\id)}
          f_* g_* \id^*
          \xrightarrow{\comp_*(g,f)}
          (fg)_* \id^*
          \xrightarrow{\triv^*}
          (fg)_*
          \bigr)
        \end{multlined}
        \\
        \label{eq:comp-triv commutation upper upper}
        \begin{multlined}[b][\linewidth - 6em]
          \bigl(
          f^* \id^* g^*
          \xrightarrow{\triv^*}
          f^* g^*
          \xrightarrow{\comp^*(f,g)}
          (gf)^*
          \bigr) \\
          = \bigl(
          f^* \id^* g^*
          \xrightarrow{\comp^*(\id,g)}
          f^* g^*
          \xrightarrow{\comp^*(f,g)}
          (gf)^*
          \bigr)
        \end{multlined}
        \\
        \label{eq:comp-triv commutation upper lower}
        \begin{multlined}[b][\linewidth - 6em]
          \bigl(
          f^* \id_* g^*
          \xrightarrow{\triv_*}
          f^* g^*
          \xrightarrow{\comp^*(f,g)}
          (gf)^*
          \bigr) \\
          = \bigl(
          f^* \id_* g^*
          \xrightarrow{\cd(\id,f;\id,f)}
          \id_* f^* g^*
          \xrightarrow{\comp^*(f,g)}
          \id^* (gf)^*
          \xrightarrow{\triv^*}
          (gf)^*
          \bigr)
        \end{multlined}
      \end{gather}
    \end{subequations}

  \item \label{en:CD-triv commutation}
    We have, referring to~\ref{eq:base change diagram}:
    \begin{subequations}
      \label{eq:CD-triv commutation}
      \begin{gather}
        \label{eq:CD-triv commutation lower}
        \begin{multlined}[b][\linewidth - 6em]
          \bigl(
          g^* \id_* f_*
          \xrightarrow{\triv_*}
          g^* f_*
          \xrightarrow{\cd(f,g;\tilde{f},\tilde{g})}
          \tilde{f}_* \tilde{g}^*
          \bigr) \\
          = \bigl(
          g^* \id_* f_*
          \xrightarrow{\comp_*(f,\id)}
          g^* f_*
          \xrightarrow{\cd(f,g;\tilde{f},\tilde{g})}
          \tilde{f}_*\tilde{g}^*
          \bigr)
        \end{multlined}
        \\
        \label{eq:CD-triv commutation upper}
        \begin{multlined}[b][\linewidth - 6em]
          \bigl(
          g^* \id^* f_*
          \xrightarrow{\triv^*}
          g^* f_*
          \xrightarrow{\cd(f,g;\tilde{f},\tilde{g})}
          \tilde{f}_* \tilde{g}^*
          \bigr) \\
          = \bigl(
          g^* \id^* f_*
          \xrightarrow{\comp^*(g,\id)}
          g^* f_*
          \xrightarrow{\cd(f,g,\tilde{f},\tilde{g})}
          \tilde{f}_* \tilde{g}^*
          \bigr)
        \end{multlined}
      \end{gather}
    \end{subequations}

  \item \label{en:CD-comp commutation}
    Consider the large commutative diagram
    \begin{equation}
      \label{eq:CD-comp commutation diagram}
      \includestandalone{img16}
    \end{equation}
    We have:
    \begin{subequations}
      \label{eq:CD-comp commutation}
      \begin{align}
        \label{eq:CD-comp commutation lower}
        \begin{multlined}[b][\linewidth - 6em]
          \bigl(
          g^* f_* h_*
          \xrightarrow{\comp_*(h,f)}
          g^* (fh)_*
          \xrightarrow{\cd(fh,g;\tilde{f}\tilde{h},k)}
          (\tilde{f}\tilde{h})_* k^*
          \bigr) \\
          = \bigl(
          g^* f_* h_*
          \xrightarrow{\cd(f,g;\tilde{f},\tilde{g})}
          \tilde{f}_* \tilde{g}^* h_*
          \xrightarrow{\cd(h,\tilde{g};\tilde{h},k)}
          \tilde{f}_* \tilde{h}_* k^*
          \xrightarrow{\comp_*(\tilde{h},\tilde{f})}
          (\tilde{f}\tilde{h})_* k^*
          \bigr)
        \end{multlined}
        \\
        \label{eq:CD-comp commutation upper}
        \begin{multlined}[b][\linewidth - 6em]
          \bigl(
          h^* f^* g_*
          \xrightarrow{\comp^*(h,f)}
          (fh)^* g_*
          \xrightarrow{\cd(g,fh;k,\tilde{f}\tilde{h})}
          k_* (\tilde{f}\tilde{h})^*
          \bigr) \\
          = \bigl(
          h^* f^* g_*
          \xrightarrow{\cd(g,f;\tilde{g},\tilde{f})}
          h^* \tilde{g}_* \tilde{f}^*
          \xrightarrow{\cd(\tilde{g},h;k,\tilde{h})}
          k_* \tilde{h}^* \tilde{f}^*
          \xrightarrow{\comp^*(\tilde{h},\tilde{f})}
          k_* (\tilde{f}\tilde{h})^*
          \bigr)
        \end{multlined}
      \end{align}
    \end{subequations}
  \end{enumerate}
\end{theorem}

\begin{proof}
  For~\ref{eq:comp-triv commutation lower lower} and~\ref{eq:comp-triv
    commutation upper upper}, we can directly
  apply~\ref{eq:trivializations comp} for the former (and its
  analogue \ref{fig:trivializations comp} for the latter), ignoring
  the second composand.  The same goes for both
  equations~\ref{eq:CD-triv commutation}.  The remainder we prove
  using string diagrams.

  For~\ref{eq:comp-triv commutation lower upper}, the string diagrams
  of the left and right transformations are, respectively:
  \begin{align}
    \label{eq:comp-triv lower upper strings}
    \includestandalone{img17a} && \includestandalone{img17b}
  \end{align}
  Clearly it is the latter diagram that needs to be simplified; to
  understand it, the blue sub-diagram is~$\cd(g,\id;\allowbreak
  g,\id)$, the red one is~$\comp_*(g,f)$,~and the black one
  is~$\triv^*$.  We recall our convention on omitting labels from
  diagrams; there should be no ambiguity provided that one recalls the
  labels of the ends.

  We begin by applying \ref{fig:trivializations comp} to the
  one~$\triv$, simplifying it to the first diagram below, which then
  transforms using \ref{fig:adjunction-composition 2} on the blue
  subdiagram:
  \begin{equation}
    \label{en:comp-triv lower upper more strings}
    \includestandalone{img18}
  \end{equation}
  Finally, we break the identity (middle lower) string according
  to \ref{fig:trivializations triv} and remove the associated
  \mbox{$\unit$-$\triv$}~combination using \ref{fig:trivializations
    adj} and then \ref*{fig:trivializations comp} again, leaving the
  first figure of~\ref{eq:comp-triv lower upper strings}, as desired.
  The same computation applies to~\ref{eq:comp-triv commutation upper
    lower} (or, to avoid repeating the same work: take the above
  computation, reverse all the arrows,~and reflect it horizontally).

  For~\ref{eq:CD-comp commutation lower}, we again render the two
  transformations as diagrams, which are somewhat more complex:
  \begin{align}
    \label{eq:CD-comp commutation strings}
    \includestandalone{img19a} && \includestandalone{img19b}
  \end{align}
  To parse the first diagram, the blue sub-diagram is~$\comp_*(h,f)$
  and the red one is~$\cd(fh,g;\allowbreak\tilde{f}\tilde{h},k)$.  To
  parse the second diagram, the blue sub-diagram
  is~$\cd(f,g;\allowbreak\tilde{f},\tilde{g})$, the red one
  is~$\cd(h,\tilde{g};\allowbreak\tilde{h},k)$,~and the black one
  is~$\comp_*(\tilde{h},\tilde{f})$.  Simplifying this requires a
  number of steps but as a first major goal we eliminate the loop.
  As usual, we use matching colors to indicate changes in the 
  diagrams, where violet denotes a subdiagram that is both blue and
  red (i.e.\ is changed both from and to the adjacent diagrams).
  \begin{equation}
    \label{eq:CD-comp commutation loop}
    \includestandalone{img20}
  \end{equation}
  In the first equality we use~\ref{fig:adjunction-composition 1},~and
  in the second, we use~\ref{SD loop}.  Now we paste this into the
  rest of~\ref{eq:CD-comp commutation strings}:
  \begin{equation}
    \label{eq:CD-comp commutation finish}
    \includestandalone{img21}
  \end{equation}
  where the first equality is \ref*{fig:compositions (assoc)} again
  and the second is~\ref{fig:adjunction-composition 1}.  The last
  diagram is the left diagram of~\ref{eq:CD-comp commutation strings},
  as desired.  The proof of~\ref{eq:CD-comp commutation upper} is the
  same (that is, with arrows reversed and the diagrams reflected
  horizontally).
\end{proof}

\subsection*{Uniqueness of the roof.}

Now we pursue a ``normal form'' for the roof similar to the
``staging'' defined for the alternating reduction.  It is much less
complicated, however.

\begin{theorem}{lem}{SGNT CD ordering}
  We have $\genby{\comp_0,\triv_0,\bc} =
  \genby{\triv_0}\genby{\comp_0}\genby{\bc}$.
\end{theorem}

\begin{proof}
  Before proceeding, note that using \ref{commutation
    relations}\ref{en:CD-triv commutation} from ``left to right''
  requires making a choice of the individual morphisms~$\tilde{f}$,
  $\tilde{g}$, and~$\tilde{h}$ given only the
  composition~$\tilde{h}\tilde{f}$; this may not, in general, be
  possible, but is in fact canonical if we assume the outer rectangle
  is cartesian (then we may take~$\tilde{g}$ to be the base change
  of~$g$).  This is why we specialize to~$\bc$ in this corollary,
  aside from the applications.
  
  Now, it follows from~\ref{horizontal composition
    verify} and the cases~\ref{en:comp-triv commutation}
  and~\ref*{en:CD-triv commutation} of~\ref*{commutation relations}
  that $\genby{\comp_0,\triv_0,\cd} =
  \genby{\triv_0}\genby{\comp_0,\cd}$, and it follows from
  case~\ref{en:CD-comp commutation} that we have $\genby{\comp_0,\cd}
  = \genby{\comp_0}\genby{\cd}$.
\end{proof}

Although this proves that any $\phi \in \genby{\comp_0, \triv_0, \bc}$
can be written as $\alpha\beta\gamma$ with $\alpha \in
\genby{\triv_0}$, $\beta \in \genby{\comp_0}$,~and $\gamma \in
\genby{\bc}$, the intermediate functors $\on{dom} \alpha = \on{cod}
\beta$ and $\on{dom}\beta = \on{cod} \gamma$ are not canonical.  In
the general case this is unfixable, though the next lemma remains
valid.  In the case of the roof, this is fortunately all that is
required.

\begin{theorem}{lem}{canonical triv}
  Let~$F$ be any~SGF and let $\phi \colon G \to F$ be in
  $\genby{\bc,\comp_0,\triv_0}$.  Then there exists an~SGF $F_{\triv}
  = F_l F F_r$ and an~SGNT $\phi_{\triv} = \phi_l F \phi_r \colon
  F_{\triv} \to F$, having the properties that:
  \begin{itemize}
  \item $F_l = \id$ if and only if $\phi_l = \id$; otherwise, $F_l =
    \id_*$ and $\phi_l = \triv_*$, and unless $F = f^* F_1$ starts
    with a~$^*$, we have $F_l = \id$,
  \item $F_r = \id$ if and only if $\phi_r = \id$; otherwise, $F_r =
    \id^*$ and $\phi_r = \triv^*$, and unless $F = F_1 f_*$ ends with
    a~$_*$, we have $F_r = \id$,
  \end{itemize}
  and a $\psi \colon G \to F_{\triv}$ in $\genby{\bc, \comp_0}$ such
  that $\phi = \phi_{\triv} \psi$.
\end{theorem}

\begin{proof}
  By~\ref{SGNT CD ordering}, it suffices to assume $\phi \in
  \genby{\triv_0}$.  By~\ref{eq:CD trivial isomorphisms lower},
  any~$\triv_*$ in the configuration $f^* \triv_* \colon f^* \id_* \to
  f^*$ can be replaced by~$\bc(f,\id)$ followed by $\triv_* f^* \colon
  \id_* f^* \to f^*$.  Inductively, then, any configuration $f_1^*
  \cdots f_n^* \triv_*$ is equal to $\triv_* f_1^* \cdots f_n^*$
  following a sequence of base change morphisms.  Similarly,
  by~\ref{eq:CD trivial isomorphisms upper} we may convert $\triv^*
  f_{1{*}} \cdots f_{n{*}}$ to $f_{1{*}} \cdots f_{n{*}} \triv^*$
  following a sequence of base changes.

  Analogously, by~\ref{eq:trivializations comp}, any~$\triv_*$ in the
  configuration~$\triv_* f_*$ or~$f_* \triv_*$ can be replaced
  wholesale with simply~$\comp_*(f,\id)$ or~$\comp_*(\id,f)$
  respectively.  Likewise for~$\triv^*$.

  Now, by~\ref{horizontal composition verify}, all~$\triv_*$
  and~$\triv^*$ morphisms~``commute'', so we may assume that those
  covered in the first paragraph occur first on the composition
  of~$\phi$, followed by those covered in the second paragraph.  It
  follows that the only trivializations that cannot be eliminated by
  this combination are those appearing in the configuration~$\triv_*
  f^*$ at the left end, or~$f_* \triv^*$ at the right end of the
  composition, giving $F_{\triv}$~and~$\phi_{\triv}$, and the above
  construction furnishes~$\psi$.
\end{proof}

Now we can prove \ref{roof morphism}.

\begin{theorem}{prop}{roof morphism uniqueness}
  For any SGF~$F$, its roof~$\roof(F)$ is the unique SGF of the
  form~$a_* b^*$ admitting admitting a map $\roof(F) \colon F \to a_*
  b^*$ in~$\SGNT_0^+$, and this map is unique.
\end{theorem}

\begin{proof}
  Let $\phi \colon F \to \roof(F)$ be in~$\SGNT_0^+$; that is, in
  $\genby{\bc,\comp_0, \triv_0}$.  \ref{SGNT CD ordering} then places
  it in the form~$\alpha \gamma \beta$ with $\alpha \in
  \genby{\triv_0}$, $\gamma \in \genby{\comp_0}$,~and $\beta \in
  \genby{\bc}$, and furthermore by~\ref{canonical triv}, $\alpha =
  \id$, since $\roof(F) = a_{F{*}} b_F^*$.

  Write $\beta \colon F \to G$; since $\gamma \colon G \to a_{G{*}}
  b_G^*$ is in~$\comp_0$, we must have~$G$ of the form $(a_{1{*}}
  \cdots a_{i{*}})(b_j^* \cdots b_1^*)$ and $\beta \in \genby{\bc}$;
  thus, the maps $a_k$~and~$b_k$ furnish the projections of the final
  object mapping to the diagram of~$F$ described in \ref{roof}, and in
  particular,~$G$ is unique. We will write $G = F_{\bc}$.

  The proof then reduces to two claims: that the map $\gamma \colon
  F_{\bc} \to \roof(F)$ is unique in~$\genby{\comp_0}$, and that the
  map $\beta \colon F \to F_{\bc}$ is unique in~$\genby{\bc}$.  The
  former is simple: since $\roof(F) = a_{F{*}} b_F^*$, the map~$\gamma$
  must be of the form $(\comp_*(?,?) \cdots)(\comp^*(?,?) \cdots)$;
  i.e.~the~$_*$ and~$^*$ compositions do not interact.  Then
  by~\ref{eq:compositions assoc}, both factors are fully associative,
  so~$\gamma$ is unique.

  For the latter, the proof follows directly from~\ref{horizontal
    composition verify} but only with the right words.  By definition,
  $\beta \in \genby{\bc}$ is a composition of base change morphisms,
  which we may view as rewriting the string of basic SGFs~$F$:
  each~$\bc(f,g)$ replaces~$g^* f_*$ with~$\tilde{f}_* \tilde{g}^*$;
  we will call the two-letter space it affects its \define{support}.
  We will say that given any particular representation of~$\beta$ as
  a composition, the basic~SGFs of~$F$ itself have \define{level}~$1$
  and that each replacement increases the level of each basic~SGF
  by~$1$.  We will say that the~\define{level} of a
  specific~$\bc(f,g)$ factor is~$n$ if that is the larger of the
  levels of $f$~and~$g$.

  By definition of level, if a factor of level~$1$ follows any other
  factor, then their supports must be disjoint, and therefore, they
  are subject to \ref*{horizontal composition verify}.
  Therefore,~$\beta$ may be written with all the factors of level~$1$
  coming first; i.e.~$\beta = \beta \beta_0$, where~$\beta_0$ is
  a composition of level-$1$ factors and~$\beta$, a composition of
  higher-level factors.

  The set of possible level-$1$ factors is the set of possible base
  change morphisms out of~$F$, each of which correspond to a
  configuration~$g^* f_*$ in~$F$.  Since in~$\roof(F)$, no such
  configurations remain, each one must be the support of some factor
  of~$\beta$, and therefore necessarily some level-$1$ factor.
  Furthermore, since their supports are disjoint their order is
  irrelevant by \ref{horizontal composition verify}.
  Therefore,~$\beta_0$ is uniquely determined by~$F$.

  Now, letting $F' = \on{cod} \beta_0$, this functor is uniquely
  determined by~$F$ and we may apply induction to $\beta \colon F'
  \to \roof(F') = \roof(F)$ to conclude that~$\beta$ is equal to a
  uniquely defined ordered product of its factors of each level, and
  is therefore unique, as claimed.
\end{proof}

The following proposition gives a less laborious construction of the
roof morphism.

\begin{theorem}{prop}{roof morphism construction}
  The roof morphism of~$F$ factors through the alternating reduction
  of~$F$ and admits a factorization $(\phi_n \beta_n) \cdots (\phi_1
  \beta_1)$, where each~$\beta_i \in \bc$ and each~$\phi_i$ is an
  alternating reduction.
\end{theorem}

\begin{proof}
  To define the roof morphism of~$F$, it suffices to define it for the
  alternating reduction~$F'$, since then by \ref{roof morphism
    uniqueness}, they will have the same roof.  We thus
  define~$\beta_1 = \id$.  Let~$\beta_2$ be any element of~$\bc$
  defined on~$F'$, and let~$\phi_2$ be the alternating reduction of
  its codomain.  We claim that we may complete the proof by induction
  applied to the codomain of~$\phi_2$.  Indeed, we have decreased the
  number of terms of~$F'$ if that number was at least three, and if it
  has exactly two terms, then the construction is completed by a
  single additional base change.
\end{proof}

We finish with a few other results extending the uniqueness of the
roof to a larger class of~SGNTs.

\begin{theorem}{cor}{SGNT0 roofs}
  If $\phi \colon F \to G$ is in $\SGNT_0$, then $\roof(F) = \roof(G)$
  as~SGFs and $\roof(G) \phi = \roof(F)$ as~SGNTs.
\end{theorem}

\begin{proof}
  For the former, write~$\phi$ as an alternating composition
  of~$\SGNT_0^+$ and its inverses.  By~\ref{roof morphism uniqueness}
  both preserve the roof (as in the proof of~\ref{good persistence}).

  For the latter, again write $\phi = \phi_0 \alpha$ with $\alpha \in
  \SGNT_0^+ \cup (\SGNT_0^+)^{-1}$ and $H = \on{dom}(\phi_0) =
  \on{cod}(\alpha)$, we have $\roof(G) \phi_0 = \roof(H)$ by
  induction.  If $\alpha \in \SGNT_0^+$ then we get $\roof(G) \phi =
  \roof(H) \alpha = \roof(F)$ by \ref*{roof morphism uniqueness}.  If
  $\alpha \in (\SGNT_0^+)^{-1}$, then we get $\roof(F) \alpha^{-1} =
  \roof(H) = \roof(G) \phi_0$.
\end{proof}

\section{Simplification via the roof}
\label{sec:simplify}

Finally we can employ the device of the roof to simplify an
arbitrary~SGNT that may contain unit morphisms.

\begin{theorem}{prop}{unit roof}
  Let $\phi \colon FG \to Ff_*f^*G$ be the SGNT $F \unit(f) G$.  Then
  there exists a map of spaces~$\tilde{f}$ and a commutative diagram:
  \begin{equation}
    \label{eq:unit roof diagram}
    \includestandalone{img22}
  \end{equation}
  where the lower edge is in~$\SGNT_0$ and is independent of~$\phi$.
\end{theorem}

\begin{proof}
  For notation, write $\roof(F) = a_{F{*}}b_F^*$, $\roof(G) = a_{G{*}}
  b_G^*$,~and $\roof(FG) = a_* b^*$ as in the statement (we used
  $a_{FG{*}} b_{FG}^*$ in \ref{roof}).  Consider the following
  diagram:
  \begin{equation}
    \label{eq:unit roof zigzag}
    \includestandalone{img23}
  \end{equation}
  where $X = \target(F)$, $Y = \source(F) = \target(G)$,~and $Z = 
  \source(G)$, the middle square is cartesian,~and $a = a_F p_F$,
  $b = b_F p_G$.  We observe the following formal identity:
  \begin{equation}
    \label{eq:unit roof compatibility}
    \roof(FG) \times_Y A \cong
    (\roof(G) \times_Y A) \times_A (A \times_Y \roof(F))
  \end{equation}
  since $\roof(FG) \cong \roof(G) \times_Y \roof(F)$ by~\ref{eq:unit
    roof zigzag}.  This is represented by the following diagram
  (rotated from the above for compactness):
  \begin{equation}
    \label{eq:unit roof compatibility diagram}
    \includestandalone{img24}
  \end{equation}
  We define~$\tilde{f}$ to be the projection $\roof(FG) \times_Y A \to
  \roof(FG)$; then, since $\roof(FG) \cong \roof(F) \times_Y
  \roof(G)$, we have the following compositions, using the projections
  from diagram~\ref{eq:unit roof zigzag}:
  \begin{align}
    \label{eq:unit roof maps}
    \tilde{f}_F \pi_F = p_F \tilde{f} &&
    \tilde{f}_G \pi_G = p_G \tilde{f}.
  \end{align}
  Using all this notation, we can define the multipart composition for
  the bottom arrow of~\ref{eq:unit roof diagram}:
  \begin{subequations}
    \label{eq:unit roof composition}
    %\allowdisplaybreaks
    \begin{align}
      \label{eq:unit roof composition 1}
      F f_* f^* G
      &\xrightarrow{\roof(F), \roof(G)}
      a_{F{*}} b_F^* f_* f^* a_{G{*}} b_G^* \\
      \label{eq:unit roof composition 2}
      &\left.\!\!\!
      \begin{aligned}
        &\xrightarrow{\bc(f,b_F), \bc(a_G, f)}
        a_{F{*}} \tilde{f}_{F{*}} \tilde{b}_F^*
        \tilde{a}_{G{*}} \tilde{f}_G^* b_G^* \\
        &\xrightarrow{\bc(\tilde{a}_G, \tilde{b}_F)}
        a_{F{*}} \tilde{f}_{F{*}}\pi_{F{*}}
        \pi_G^* \tilde{f}_G^* b_G^* \\
        &\xrightarrow{\comp_*(\pi_F, \tilde{f}_F), \comp^*(\pi_G,
          \tilde{f}_G)}
        a_{F{*}} (\tilde{f}_F \pi_F)_* (\tilde{f}_G \pi_G)^* b_G^* \\
        &=
        a_{F{*}} (p_F \tilde{f})_* (p_G \tilde{f})^* b_G^* \\
        &\xrightarrow{\comp_*(\tilde{f}, p_F)^{-1}, \comp^*(\tilde{f},
          p_G)^{-1}}
        a_{F{*}} p_{F{*}} \tilde{f}_* \tilde{f}^* p_G^* b_G^*
      \end{aligned}
      \right\}
      \\
      \label{eq:unit roof composition 3}
      &\xrightarrow{\comp_*(p_F, a_F), \comp^*(p_G, b_G)}
      a_* \tilde{f}_* \tilde{f}^* b^*.
    \end{align}
  \end{subequations}
  We have braced the middle lines for comparison with~$\roof(FG)$,
  which by~\ref{SGNT0 roofs} may be written as:
  \begin{subequations}
    \label{eq:roof composition}
    \begin{align}
      \label{eq:roof composition 1}
      FG
      &\xrightarrow{\roof(F), \roof(G)}
      a_{F{*}} b_F^* a_{G{*}} b_G^* \\
      \label{eq:roof composition 2}
      &\xrightarrow{\bc(a_G, b_F)}
      a_{F{*}} p_{F{*}} p_G^* b_G^* \\
      \label{eq:roof composition 3}
      &\xrightarrow{\comp_*(p_F, a_F), \comp^*(p_G, b_G)}
      a_* b^*.
    \end{align}
  \end{subequations}
  To show that~\ref{eq:unit roof diagram} commutes with~\ref{eq:unit
    roof composition} as the lower edge, we have to show that (using
  the numbers as names) $\ref*{eq:unit roof composition} \circ
  F\unit(f)G = a_* \unit(\tilde{f}) b^* \circ \ref*{eq:roof
    composition}$.  According to~\ref{horizontal composition verify},
  we have both of:
  \begin{gather}
    \label{eq:unit roof commutation}
    \ref*{eq:unit roof composition 1} \circ F\unit(f)G
    = a_{F{*}} b_F^* \unit(f) a_{G{*}} b_G^*  \circ
    \ref*{eq:roof composition 1} \\
    \ref*{eq:unit roof composition 3} \circ
    a_{F{*}} p_{F{*}} \unit(\tilde{f}) p_G^* b_G^*
    = a_* \unit(\tilde{f}) b^* \circ
    \ref*{eq:roof composition 3}
  \end{gather}
  so since $\ref*{eq:unit roof composition} = \ref*{eq:unit roof
    composition 3} \ref*{eq:unit roof composition 2} \ref*{eq:unit
    roof composition 1}$ and $\ref*{eq:roof composition} =
  \ref*{eq:roof composition 3} \ref*{eq:roof composition 2}
  \ref*{eq:roof composition 1}$ it suffices to show that:
  \begin{equation}
    \label{eq:unit roof sufficient}
    \ref*{eq:unit roof composition 2}
    \circ a_{F{*}} b_F^* \unit(f) a_{G{*}} b_G^* \\
    = a_{F{*}} p_{F{*}} \unit(\tilde{f}) p_G^* b_G^*
    \circ \ref*{eq:roof composition 2}
  \end{equation}
  We can omit the $a_{F{*}}$~and~$b_G^*$ on the ends and move the
  $\comp_*$~and~$\comp^*$ inverses in~\ref{eq:unit roof composition 2}
  to the other side, rendering both sides as maps of~SGFs
  \begin{equation}
    \label{eq:unit roof signature}
    b_F^*a_{G{*}} \to (p_F\tilde{f})_* (p_G \tilde{f})^*.
  \end{equation}
  We show these are equal using string diagrams.  First, the two sides
  of~\ref{eq:unit roof sufficient} are:
  \begin{align}
    \label{eq:unit roof sufficient strings}
    \includestandalone{img25a} && \includestandalone{img25b}
  \end{align}
  Note that $\tilde{f}_F \pi_F = p_F \tilde{f}$ and the same for~$G$,
  by~\ref{eq:unit roof maps}.  In the left diagram, the blue portion
  is~$\unit(f)$; the red portion is~$\bc(f,b_F)\bc(a_G,f)$; the brown
  portion is $\bc(\tilde{a}_G, \tilde{b}_F)$; and the yellow portion
  is $\comp_*(\pi_F, \tilde{f}_F) \comp^*(\pi_G, \tilde{f}_G)$.  In
  the right diagram, the blue is~$\unit(\tilde{f})$; the red
  is~$\bc(a_G, b_F)$; and the brown is $\comp_*(\tilde{f},
  p_F)\comp^*(\tilde{f}, p_G)$.

  Despite the complexity of these diagrams we claim that both are
  equivalent to that of $\cd(a_G, b_F; p_G \tilde{f}, p_F \tilde{f})$.
  First, the second one, where we match blue and red in consecutive
  pictures to track regions that are altered; violet means a shape
  that is both blue and red.
  \begin{equation}
    \label{eq:unit roof second strings}
    \includestandalone{img26}
  \end{equation}
  by \ref{SD loop}, and this is exactly the desired $\cd$ diagram.
  For the larger diagram we have to do only scarcely more:
  \begin{equation}
    \label{eq:unit roof first strings}
    \includestandalone{img27}
  \end{equation}
  We have used~\ref{fig:adjunctions} on the blue diagram (with the
  cyan diagram unchanged for comparison),
  and~\ref{fig:adjunction-composition 1} on the red diagram.  This
  rather extended result is now amenable to~\ref{SD loop} applied
  twice:
  \begin{equation}
    \label{eq:unit roof first strings loops}
    \includestandalone{img28}
  \end{equation}
  where, finally, we have used~\ref{fig:compositions (inv)} on the
  second diagram.  The third is once again a~$\cd$ diagram,
  necessarily $\cd(a_G, b_F; p_G \tilde{f}, p_F \tilde{f})$ because
  the ends are correct.  This completes the~proof.
\end{proof}

\begin{theorem}{prop}{SGNT0 square}
  Let $\phi \colon F \to G$ be in $\genby{\SGNT_0, \unit}$; then there
  exists a map of spaces~$g$ making the following diagram commute:
  \begin{equation}
    \label{eq:positive SGNT diagram}
    \includestandalone{img29}
  \end{equation}
\end{theorem}

\begin{proof}
  We apply induction on~$\phi$: thus, suppose that $\phi = \alpha
  \phi_0$, where diagram~\ref{eq:positive SGNT diagram} exists for
  $\phi_0 \colon F \to F'$ and $\alpha \in \SGNT_0 \cup \unit$.  If
  $\alpha \in \SGNT_0$, then we can augment the $\phi_0$~diagram
  simply:
  \begin{equation}
    \label{eq:positive SGNT diagram induction 1}
    \includestandalone{img30}
  \end{equation}
  where $\roof(F') = \roof(G)$ by \ref{SGNT0 roofs}; the triangle
  commutes by \ref{SGNT0 roofs}.  If, alternatively, $\alpha \in
  \unit$, then we write $F' = AB$ and augment the $\phi_0$~diagram
  with~\ref{eq:unit roof diagram}:
  \begin{equation}
    \label{eq:positive SGNT diagram induction 2}
    \includestandalone{img31}
  \end{equation}
  The lower edge is, by~\ref{SGNT0 roofs}, equal to~$\roof(G)$;
  since both squares commute and the triangle commutes by
  construction, the large diagram commutes.  We claim that the right
  edge is equal to $(\comp_*(gf,a)\comp^*(gf,b)) \circ a_* \unit(gf)
  b^*$.  We prove this using string diagrams:
  \begin{equation}
    \label{eq:positive SGNT diagram induction 2 strings}
    \includestandalone{img32}
  \end{equation}
  where we have used first~\ref{fig:compositions (assoc)} and
  then~\ref{eq:composition units counits} from the proof
  of~\ref{fig:adjunction-composition 1}.
\end{proof}

This is the ultimate theorem for unit morphisms alone; now we extend
it to include inverse units.

\begin{theorem}{lem}{Unit square}
  In~\ref{unit roof}, if~$\phi$ is in~$\Unit$, then so is the unit in
  the right edge.
\end{theorem}

\begin{proof}
  For concurrency of notation, replace $F$~and~$G$ in~\ref{eq:unit
    roof diagram} with $FA$~and~$BG$ respectively, where we assume as
  in Definitions~\plainref{acyclicity structure}
  and~\plainref{invertible units} that $A\unit(f)B$ is a natural
  isomorphism with $AB$~admissible and~$\phi$ itself good.  Let
  $\psi = a_* \unit(\tilde{f}) b^*$ for brevity.

  As usual, we write $\roof(F) = a_{F{*}} b_F^*$ and $\roof(G) =
  a_{G{*}} b_G^*$, and similarly $\roof(AB) = a_{AB{*}} b_{AB}^*$; by
  hypothesis, we have $(a_{AB}, b_{AB})$ admissible.  Finally, we
  write $\roof(b_F^* a_{AB{*}}\allowbreak b_{AB}^* a_{G{*}}) =
  a_{0{*}} b_0^*$, so by \ref{SGNT0 roofs} (together with \ref{roof
    morphism}) we have the alternating reduction $\roof(a_{F{*}}
  a_{0{*}} b_0^* b_G^*) = \roof(FABG) = a_* b^*$.  We claim that $a_*
  \unit(\tilde{f}) b^*$ is in~$\Unit$.

  First, we verify that~$a_* \unit(\tilde{f}) b^*$ is good.  Indeed,
  we have already shown that $a_* b^* = \roof(FABG)$, which is good by
  hypothesis on~$\phi$ and~\ref{roof isomorphism}; likewise, the
  alternating reduction of $a_* \tilde{f}_* \tilde{f}^* b^*$ is equal
  to~$\roof(FAf_*f^*BG)$ by the same token and the lower edge
  of~\ref{unit roof}, so is also good.

  Next, we verify that~$a_{0{*}} \unit(\tilde{f}) b_0^*$ is an
  isomorphism.  Indeed, if we write the diagram corresponding
  to \ref*{unit roof} with $F \leftrightarrow b_F^* A$ and $G
  \leftrightarrow B a_{G{*}}$:
  \begin{equation}
    \label{eq:Unit square diagram}
    \includestandalone{img33}
  \end{equation}
  then it appears as the right edge.  Both horizontal edges are good,
  as the alternating reduction of~$b_F^* AB a_{G{*}}$ is that
  of~$FABG$ with $a_{F{*}}$~and~$b_G^*$ removed, so by~\ref{roof
    isomorphism} they are isomorphisms.  Since the left edge
  contains~$A\unit(f)B$, which is an isomorphism by hypothesis, the
  right edge is an isomorphism, as claimed.

  Finally, we verify that the pair map~$(a_0,b_0)$ is admissible.
  Indeed, if we write the diagram
  \begin{equation}
    \label{eq:Unit_1 isomorphism zigzag}
    \includestandalone{img34}
  \end{equation}
  then the map from its roof to the product of its two projections
  $X$~and~$Z$ is
  \begin{equation}
    \label{eq:Unit_1 isomorphism roof}
    (b_0, a_0)
    = \bigl(X \times_U Y \times_V Z \to X \times Z\bigr)
    = X \times_U \bigl(Y \to U \times V\bigr) \times_V Z
  \end{equation}
  and is therefore the base change of an admissible map, so
  admissible.
\end{proof}

\begin{theorem}{lem}{unit fractions swap}
  Let $F = a_* b^*$ and let $\phi = a_* \unit(f) b^*$ be a left~or
  right isomorphism.  Then for any $\psi = a_* \unit(g) b^*$, it is
  possible to write $\psi\phi^{-1} = \alpha^{-1}\beta$ for some
  $\alpha, \beta$ with $\alpha$~invertible.
\end{theorem}

\begin{proof}
  The proof is drawn from the~``calculus of
  fractions'';~$\psi\phi^{-1}$ is represented by the upper-left corner
  of the following two diagrams, and we take $\alpha$~and~$\beta$ to
  be the other two edges in one of them:
  \begin{align}
    \label{eq:unit-unit overlap}
    \includestandalone{img36a} && \includestandalone{img36b}
  \end{align}
  We choose depending on whether it is $a_* \unit(f)$~or~$\unit(f)
  b^*$ that is an isomorphism; this ensures that the same portion
  of~$\alpha$ is also an isomorphism.
\end{proof}

Now we can give the proof of our main technical theorem, which we
restate for clarity.

\begin{theorem}{prop}{SGNT0 inverses square'}
  Let $\phi \colon F \to G$ be in~$\genby{\SGNT_0, \unit ,\Unit^{-1}}$,
  and denote $\roof(G) = a_{G{*}} b_G^*$.  Then there exist maps of
  spaces $f$~and~$g$, such that~$a_{G{*}} \unit(g) b_G^*$ is a natural
  isomorphism, forming a commutative diagram:
  \begin{equation}
    \label{eq:SGNT0 inverses square'}
    \includestandalone{img37}
  \end{equation}
\end{theorem}

\begin{proof}
  As in the statement of the theorem, our convention in this proof will
  be to draw the inverses of invertible units as arrows pointing the
  wrong way (there, down; here, left).

  The proof is by induction on the length of~$\phi$ as an alternating
  composition of $\genby{\unit,\SGNT_0}$~and~$\Unit^{-1}$.  If it has
  only one factor, then the~theorem follows from~\ref{SGNT0 square} in
  the former case, and from~\ref{unit roof} (upside-down)
  and~\ref{Unit square} in the latter.  Thus, suppose~$\phi$ has at
  least two factors, and write $\phi = \phi_0 \phi_1$, with~$\phi_0$
  having fewer factors and~$\phi_1$ being a single factor.  By
  induction we can form diagram~\ref{eq:SGNT0 inverses square'}
  for~$\phi_0$ and~\ref{eq:positive SGNT diagram} for~$\phi_1$, giving
  the following diagram, which we draw rotated to save space:
  \begin{equation}
    \label{eq:SGNT0 inverses induction square}
    \includestandalone{img38}
  \end{equation}
  Here, $H$~is some intermediate~SGF.  If $\phi_1 \in \genby{\unit,
    \SGNT_0}$, then so is~$\psi$ and therefore~$\beta_0 \psi$, and
  therefore we can apply diagram~\ref*{eq:positive SGNT diagram} to
  both halves, giving a larger diagram:
  \begin{equation}
    \label{eq:SGNT0 inverses comp square}
    \includestandalone{img39}
  \end{equation}
  which is what we want.  Now, suppose that $\phi_1 \in \Unit^{-1}$;
  then we rewrite the top line of~\ref{eq:SGNT0 inverses induction
    square} as
  \begin{multline}
    \label{eq:SGNT0 inverses top line}
    a_{F{*}} b_F^*
    \xleftarrow[\phi_1^{-1}]
    {\comp_*(f,a_F)\comp^*(f,b_F) \circ a_{F{*}}\unit(f)b_F^*}
    (a_F f)_* (b_F f)^* = a_{H{*}} b_H^* = \\
    = (a_K g)_* (b_K g)^*
    \xrightarrow[\beta_0]
    {\comp_*(g,a_K)\comp^*(g,b_K) \circ a_{K{*}}\unit(g)b_K^*}
    a_{K{*}} b_K^*
    \xrightarrow{\alpha_0}
    \roof(G);
  \end{multline}
  Leaving the~$\comp$s on the outside, the two units form the
  combination considered in~\ref{unit fractions swap}, where by
  \ref{Unit square}, we have $a_{F{*}} \unit(f) b_F^* \in \Unit$ and so,
  by definition of acyclicity structure, is either a left-~or
  right-isomorphism.  Therefore we can replace them with two different
  elements of~$\unit$ (with a common target different from $H$).  Since
  the~$\comp$s are invertible, that means that we can
  rewrite~\ref*{eq:SGNT0 inverses induction square} as
  \begin{equation}
    \label{eq:SGNT0 inverses comp square 2}
    \includestandalone{img40}
  \end{equation}
  The second diagonal arrow is, as indicated, invertible by \ref*{unit
    fractions swap}.  Then, as before, we may apply~\ref{SGNT0 square}
  to the left and to the composition of the two right arrows on the
  top of this diagram to complete the proof.
\end{proof}

\section{Proofs of the main theorems}
\label{sec:proofs}
Here are the proofs of the remaining main results and supporting lemmas.

\subsection*{Proof of \ref{counits auxiliary}.}
Let $\unit'(f) \colon f^* f_* \to \id$ be a counit morphism, and
consider the following diagram:
\begin{equation}
  \label{eq:diagonal products}
  \includestandalone{img41}
\end{equation}
Then, in short, we have the following sequence of maps whose
composition is an~SGNT $f^* f_* \to \id$ in~$\genby{\SGNT_0^+,
  \unit}$.
\begin{multline}
  \label{eq:counit composition}
  f^* f_*
  \xrightarrow{\bc(f,f)}
  \tilde{f}_{1{*}} \tilde{f}_2^*
  \xrightarrow{(\tilde{f}_1)_*\unit(\Delta)\tilde{f}_2^*}
  \tilde{f}_{1{*}} \Delta_* \Delta^* \tilde{f}_2^* \\
  \xrightarrow{\comp_*(\Delta,\tilde{f}_1)\comp^*(\Delta,\tilde{f}_2)}
  (\tilde{f}_1 \Delta)_* (\tilde{f}_2 \Delta)^* \\
  =
  \id_* \id^*
  \xrightarrow{\triv_*\triv^*}
  \id\, \id
  = \id.
\end{multline}
To see that this coincides with~$\unit'(f)$, we do a string diagram
computation.  Below is the diagram of the map constructed in
\ref{eq:counit composition}:
\begin{equation}
  \label{eq:counit strings}
  \includestandalone{img42}
\end{equation}
where the red portion is~$\bc(f,f)$, the blue portion
is~$\unit(\Delta)$, the brown portion is
$\comp_*(\Delta,\tilde{f}_1)\comp^*(\Delta,\tilde{f}_2)$,~and the
black portion is~$\triv_* \triv^*$.  This is precisely the second
diagram considered in~\ref{eq:unit roof sufficient strings}, with
$a_G$~and~$b_F$ replaced by~$f$ and the upper ends replaced by
$\id^*$~and~$\id_*$ and two~$\triv$s applied.  Accounting for the
change in notation, diagram~\ref{eq:counit strings} is equivalent to
$\triv_*\triv^* \circ \cd(f,f;\id,\id)$:
\begin{equation}
  \label{eq:counit strings final}
  \includestandalone{img43}
\end{equation}
By~\ref{fig:trivializations adj} and~\ref{fig:trivializations
  comp}, this becomes merely~$\unit'(f)$, as desired.

Suppose now that $\psi = F\phi G$ is good, where $\phi = A \unit'(f)
B$ an isomorphism,~$\unit'(f)$ is good,~and $\on{dom}(\phi) = Af^*f_*B$
admissible as in \ref{invertible units}.  Since~$f^* f_*$ is good, the
factor~$\bc(f,f)$ is an isomorphism by \ref{geolocalizing}.  Thus, the
composition~\ref{eq:counit composition} contains only one potentially
non-isomorphism, namely the term $A \tilde{f}_{1{*}} \unit(\Delta)
\tilde{f}_2^* B$, which it follows is an isomorphism as well.  It is
good, even after composing with $FA$~and~$BG$: for its domain $F
Af_{1{*}} f_2^*B G$, this follows from \ref{roof isomorphism} and
\ref{roof morphism uniqueness}, since $\bc(f,f) \in \SGNT_0^+$; for
its codomain, the entire trailing part of the diagram is its partial
alternating reduction to $FABG$, which is assumed to be good.
Finally, by uniqueness of the roof from~\ref{roof morphism}:
\begin{equation}
  \label{eq:counit roof}
  \roof(A \tilde{f}_{1{*}} \tilde{f}_2^* B)
  =
  \roof(Af^* f_*B)
  = \roof(\on{dom} \phi),
\end{equation}
so $Af_{1{*}}f_2^*B$ is admissible.  Thus, $A f_{1{*}}
\unit(\Delta) f_2^* B \in \Unit$, so $\psi \in \genby{\SGNT_0^+,
  \Unit}$, as claimed. \qed

\subsection*{Proof of \ref{secondary theorem}.}
This follows from examining \ref{SGNT0 square}.  Clearly both the
upper and lower edges are either the identity or a single
trivialization each, while the right edge must be the identity since
any~$\unit(f)$ would incur both a~$_*$ and a~$^*$ in~$G$, not both of
which are present.

\subsection*{Proof of \ref{main theorem}.}
By \ref{counits auxiliary} we may use~$\genby{\SGNT_0, \unit}$ in
place of~$\genby{\SGNT \cup\allowbreak \bc^{-1}}$.  We show uniqueness by
applying~\ref{SGNT0 inverses square'}; if the bottom edge is a natural
isomorphism then it suffices to show that the right edge is
independent of~$\phi$.  We assume that both arrows in this edge occur;
the case in which only one does is treated in~\ref{conditions
  auxiliary}.  We denote the right edge by~$\alpha^{-1} \beta$.

Write $X = \source(F) = \source(G)$ and $Y = \target(F) = \target(G)$,
and let $A = \roof(F)$ and $B = \roof(G)$, with projections $a_F
\colon A \to X$, $b_F \colon A \to Y$, and similarly for~$G$. Both
maps $f$~and~$g$ necessarily have the same source~$C$; we have $f
\colon C \to A$ and $g \colon C \to B$.  In order for $a_F f = a_G g$
and $b_F f = b_G g$, it is equivalent that the composites $(b_F, a_F)
f = (b_G, a_G) g$ into~$X \times Y$ be equal.  Such a pair of maps $C
\to A, B$ is equivalent once again to a single map $h \colon C \to A
\times_{X \times Y} B$.  Let $a$~and~$b$ be the two projections of
this fibered product.

We have $f = ah$ and $g = bh$, so by diagram~\ref{eq:composition units
  counits} in the proof of \ref{fig:adjunction-composition 1}, we have
\begin{equation}
  \label{eq:full theorem units}
  \unit(f) = \bigl(
  \id
  \xrightarrow{\unit(a)}
  a_* a^*
  \xrightarrow{a_* \unit(h) a^*}
  a_* h_* h^* a^*
  \xrightarrow{\comp_*(h,a) \comp^*(h,a)}
  f_* f^*
  \bigr)
\end{equation}
and similarly for~$g$.  After composing with $a_{F{*}}$~and~$b_F^*$
(resp.~$a_{G{*}}$ and~$b_G^*$), applying $\comp_*(f,a_F)
\comp^*(f,b_F)$ to the end is the same as the following,
by~\ref{horizontal composition verify}:
\begin{equation}
  \label{eq:full theorem composition}
  a_{F{*}} b_F^*
  \to a_{F{*}} a_* a^* b_F^*
  \to (a_F a)_* (b_F a)^*
  \to (a_F a)_* h_* h^* (b_F a)^*
  \to (a_F a h)_* (b_F a h)^*,
\end{equation}
and similarly for~$g$, with~$G$ replacing~$F$ and~$b$ replacing~$a$.
In the latter situation, assuming that $a_{G{*}} \unit(b) b_G^*$ is an
isomorphism, so is $(a_G b)_* \unit(h) (b_G b)^*$ as the only
potentially non-isomorphism in~\ref{eq:full theorem composition}, and
since $a_F a = a_G b$ and $b_F a = b_G b$, the last two steps of both
are identical and so cancel out in~$\alpha^{-1} \beta$.  Thus, we may
assume $f = a$ and $g = b$, which are uniquely determined by
$F$~and~$G$, making~$\phi$ canonical. \qed

\subsection*{Proof of \ref{conditions auxiliary}.}
It follows from~\ref{roof isomorphism} that $\roof(G)$ is an
isomorphism if~$G$ is good.

For the second condition, first note that if $\phi \in \genby{\SGNT_0,
  \unit}$, then by~\ref{eq:positive SGNT diagram},~$(b_G, a_G)$
factors through~$(b_F, a_F)$.  Assuming that factorization, we have a
graph morphism $\roof(G) \to \roof(F) \times_Z \roof(G)$ forming a
section of~$b$.  Since $(b_F, a_F)$ is a universal monomorphism,
so~$b$ is a monomorphism, and therefore that section is an
isomorphism; thus, $a_{G{*}} \unit(b) b_G^*$ (in fact,~$\unit(b)$
itself) is an isomorphism.

Observe that we can, in this case, fill in an~$\alpha^{-1}$ to go with
$\beta = \phi$, in the notation of the proof of \ref*{main theorem}.
Namely, we take $\alpha = \id = \comp_*(\id,a_F)\comp^*(\id,b_F) \circ
\unit(\id)$, so the formal setup of the previous proof applies.

If we have a $\phi \in \genby{\SGNT_0, \Unit^{-1}}$, then by diagram
\ref*{eq:positive SGNT diagram} taken upside down, we find that~$(b_F,
a_F)$ factors through~$(b_G, a_G)$ by some map~$g$; when~$G$ is weakly
admissible, the same argument applies and shows that $\roof(F)
\times_Z \roof(G) \cong \roof(F)$, with the projection onto~$\roof(G)$
being~$g$.  Thus the right edge of the diagram is~$a_{G{*}} \unit(b)
b_G^*$, and is also invertible.  As before, we can assume that $\phi =
\alpha^{-1}$ is complemented by a trivial~$\beta$ for notational
purposes.\qed

\subsection*{Proof of \ref{etale acyclicity}.}
The SGA4 result that we require is the following criterion for an
invertible unit morphism; we assume the same hypotheses on schemes as
in the description of the \'etale context.

\begin{theorem}{lem}{SGA thm}
  (\cite{sga4}*{Exp.~xv, Th.~1.15}) Let $f \colon X \to Y$ be
  separated and of finite type, as well as locally acyclic (for
  example, smooth).  Let~$\sh{F}$ be an $\ell$-torsion (or, therefore,
  $\ell$-adic) sheaf on~$Y$.  Then the unit morphism of sheaves
  \begin{equation}
    \label{eq:SGA thm}
    \unit(f)_{\sh{F}} \colon \sh{F} \to f_* f^* \sh{F}
  \end{equation}
  is an isomorphism if and only if, for every algebraic geometric
  point $g \colon y \to Y$ with fiber $f_y \colon X_y \to y$, the unit
  morphism
  \begin{equation}
    \label{eq:SGA thm fiber}
    \unit(f_y)_{g^* \sh{F}} \colon g^* \sh{F} \to f_{y{*}}^{} f_y^* g^* \sh{F}
  \end{equation}
  is an isomorphism. \qed
\end{theorem}

Now we proceed to the proof.  There are three statements to verify, of
which two are trivial:
\begin{itemize}
\item If~$i$ is an isomorphism, then for any morphisms $a$~or~$b$, the
  pair map $(a,i)$ or $(i,b)$ is isomorphic to the graph of
  $a$~or~$b$, which is an immersion.
\item The base change of any immersion is again an immersion.
\end{itemize}

For the third statement, we must verify that if $\phi =
a_*\unit(f)b^*$ is good and a natural isomorphism with $(a,b)$ an
immersion, then~$\phi$ is a left or right isomorphism.  The goodness
hypothesis entails that either both $af$~and~$a$ are proper, or
$bf$~and~$b$ are smooth.  Let us write $a \colon Y \to A$ and $b
\colon Y \to B$, so $f \colon X \to Y$ as in the statement of \ref{SGA
  thm}.

Consider the proper case, and let~$p$ be any (geometric) point
of~$A$.  We will show that the stalk of~$a_* \unit(f)$ at~$p$ is an
isomorphism, and therefore that~$a_* \unit(f)$ is itself an
isomorphism since~$p$ is arbitrary.  Let~$\tilde{p}$ denote the
fiber of~$a$ over~$p$ and let~$\tilde{p}_f$ denote that of~$af$
over~$p$.  Using~\ref{unit roof} on $p^* a_*$~and~$p^* (af)_*$, the
stalk of~$a_* \unit(f) b^*$ at~$p$ is the unit map from~$p^* a_* b^*
\isoarrow a|_{p{*}} (b \tilde{p})^*$ to
\begin{multline}
  \label{eq:acyclicty lemma proper}
  p^* a_* f_* f^* b^*
  \cong p^* (af)_* f^* b^*
  \isoarrow (a|_p f|_{\tilde{p}})_* \tilde{p}_f^* f^* b^* \\
  \cong a|_{p{*}} f|_{\tilde{p}{*}}(f\tilde{p}_f)^* b^*
  \cong a|_{p{*}} f|_{\tilde{p}{*}}(\tilde{p}f|_{\tilde{p}})^* b^* \\
  \cong a|_{p{*}} f|_{\tilde{p}{*}}f|_{\tilde{p}}^* \tilde{p}^*b^*
  \cong a|_{p{*}} f|_{\tilde{p}{*}}f|_{\tilde{p}}^* (b\tilde{p})^*,
\end{multline}
where by~\ref{base change theorems} the arrows are isomorphisms
since $a$~and~$af$ are proper.  Since~$(a,b)$ is an immersion,
$b\tilde{p}$~is also an immersion (the base change of~$(a,b)$
along~$p$) and therefore $(b\tilde{p})^* (b\tilde{p})_* \isoarrow
\id$.  Applying~$(b\tilde{p})_*$ to the right above, we find that
$a|_{p{*}} \unit(f|_{\tilde{p}})$ is an isomorphism.  The same
computation shows that this is the stalk of $a_* \unit(f)$ at $p$,
as desired.

Consider the smooth case.  Then for any point~$q$ of~$B$ we again
have isomorphisms
\begin{align}
  \label{eq:acyclicity lemma smooth}
  a_* b^* q_* \isoarrow (a\tilde{q})_* b|_q^* &&
  a_* f_* f^* b^* q_*
  \isoarrow (a\tilde{q})_* f|_{\tilde{q}{*}} f|_{\tilde{q}}^* b|_q^*
\end{align}
in which $a\tilde{q}$ is an immersion and thus $(a\tilde{q})^*
(a\tilde{q})_* \isoarrow \id$.  Applying that pullback, we find that
$\unit(f|_{\tilde{q}}) b|_q^*$ is an isomorphism.  Since~$f$ is
smooth, it is locally acyclic, so by~\ref{SGA thm} all its fibers
$\unit(f|_p) p^* b|_q^* = \unit(f|_p) b|_p^*$ are isomorphisms, over
all points~$p$ of~$X$.  Applying it again, this means that~$\unit(f)
b^*$ is an isomorphism.

\section{Comments and acknowledgements}
\label{s:comments}

Owing to the high level of abstraction in our presentation and the
precise formulation of our definitions and theorems, some analysis of
the limitations of this line of investigation is in order.

\subsection*{Comments and counterexamples.}
The conditions of~\ref{main theorem} may require some explanation.
Invertibility of the roof morphism is of course technically necessary
in the proof, and the ``good'' property of~\ref{conditions auxiliary}
gives convenient access to it, but some such condition is actually
necessary, as the following example due to Paul Balmer shows:

\begin{theorem}{ex}{balmer}
  Let~$X$ be the scheme $(\Aff^1 \setminus \{0\}) \sqcup \{0\}$; that
  is, the affine line with the origin detached, and let $f \colon X
  \to \Aff^1$ be the natural map that is the identity on each
  connected component of~$X$.  There are two SGNTs from~$f^* f_* f^*$
  to itself: the identity map, and the composition $\phi = f^*
  \unit(f) \circ \unit'(f) f^*$.  They are not equal, as can be seen
  by computing them on the constant sheaf~$\sh{C}$ of rank~$1$
  on~$\Aff^1$ (this works for any kind of sheaf):
  \begin{itemize}
  \item $f^* \sh{C}$~is again the constant sheaf; $f_* f^* \sh{C}$~has
    rank~$2$ on every neighborhood of~$\{0\}$; therefore $f^* f_* f^*
    \sh{C}$~has rank~$2$ on~$\{0\}$.
  \item The map~$\unit'(f)f^*$ already has to map something of
    rank~$2$ to something of rank~$1$, so is not injective;
    therefore~$\phi$ cannot be an isomorphism, much less the identity.
  \end{itemize}
  Since~$f$ is neither proper nor even flat, of course~\ref{conditions
    auxiliary} does not apply; this example illustrates the necessity
  of gaining control of the pathologies of the maps along which the
  functors are taken.  In fact, the map~$\roof(F)$ is not an
  isomorphism either: we have $\roof(F) = \tilde{f}_{2*}
  (f\tilde{f}_1)^*$, where~$\tilde{f}_i$ are the projections of $X
  \times_{\Aff^1} X$ onto~$X$, and one can see that, applied
  to~$\sh{C}$ on~$\Aff^1$, it yields a sheaf on~$X$ with rank~$4$
  at~$\{0\}$ and rank~$2$ elsewhere, which is nowhere isomorphic
  to~$F\sh{C}$ as computed above.
\end{theorem}

Our second comment concerns the specific and careful definition of the
class~$\Unit$.  The ultimate goal was to be able to prove~\ref{unit
  fractions swap}, which requires only the property of being a ``left
or right isomorphism'' (see~\ref{acyclicity structure}) but whose
partner~\ref{Unit square} was easily proven only for~SGNTs of the
simple form allowed by the ``trivial'' acylicity structure (this is
actually a simplified version of the very involved history of this
research).  We felt that more general invertible ``units'' were likely
to occur in reality, and eventually arrived at the statement of
\ref{etale acyclicity}, which is the key ingredient in the expanded
class, by pondering the following example:

\begin{theorem}{ex}{cohomology not isomorphism}
  Let $X = \Aff^1$ and let $f \colon \{0\} \to X$ and $g \colon \{1\}
  \to X$ be the closed immersions of two points, and
  consider~\ref{eq:unit-unit overlap}.  Denote by~$p$ the map from~$X$
  to a point, and let $a = b = p$.  Then although both squares commute
  and their common left vertical arrow is an isomorphism, their right
  vertical arrows are both zero.
\end{theorem}

Of course, in this example, the map $(p, p) \colon X \to \mathrm{pt}
\times \mathrm{pt}$ is far from an immersion.  But it illustrates how
the failure of this condition can cause problems, morally speaking by
subtracting information from the unit morphism embedded between
$a_*$~and~$b^*$ to the point that it becomes an isomorphism when it
should not; note that in \ref*{cohomology not isomorphism}, the map
$p^* \to f_* f^* p^*$ alone is very much not an isomorphism.  This map
corresponds to the actual immersion $(\id, p) \colon X \to X \times
\mathrm{pt}$, in which the key \ref{unit fractions swap} actually does
hold.

As for the other condition of \ref{main theorem}, we have no
particular insight into its general meaning, but we do note that even
for maps $a_* b^* \to c_* d^*$, the~theorem can fail if we have $(c,d)
= (a,b)g$ for multiple maps~$g$, giving not necessarily equal~SGNTs
$a_* \unit(g) b^*$.  This is forbidden by the weak admissibility
hypothesis of \ref{conditions auxiliary}.

Finally, we comment on our choice of terminology for ``standard
geometric functors''.  It is easy to imagine trying to prove theorems
similar to the above involving not only $f_*$~and~$f^*$ but also
$f_!$~and~$f^!$ (the ``exceptional'' pushforward and pullback), and
indeed, this was the original intention of this paper.  Unfortunately,
we were unable to identify the correct context for such results; it
seems likely that they will need to include, as well, the bifunctors
$\on{Hom}$~and~$\otimes$, filling out the full complement of the six
functors, in order to adequately express the relationship between
$f^!$~and~$f^*$.  Furthermore, the techniques of this paper appear
inadequate, as a functor such as~$f_* g_!$ is not alternating but
apparently has no alternating reduction (hence no roof), and is
seemingly incomparable with~$f_! g_*$.

\subsection*{Acknowledgements.}

This paper would probably not have been finished were it not for Mitya
Boyarchenko's encouragement and his astute, if disruptive, reading of
the first draft and its several implicit errors.  I am also grateful
to Brian Conrad for commenting on that draft and for advocating the
next section, whose title was another of Mitya's suggestions.  During
the sophomoric stages of this research, Paul Balmer was very generous
in giving me much seminar time for it, as well as disproving the
original theorem and, therefore, motivating my formulation of
everything that is now in the paper.  Finally, I want to thank the
anonymous referee for requesting additional organizational clarity and
precision of language that led to my formulation of \ref{geofibered
  category}.

\appendix
\section{User guide}
\label{sec:guide}

This section is an informal description of the intuition and use of
the main theorem \ref{main theorem} aimed at readers hoping to find a
connection with familiar appearances of the so-called geometric
functors.  We begin with a non-rigorous reformulation of our definitions:

\begin{theorem*}{defn}
  (Imprecise) Let~$F$ and~$G$ be two functors of the form~$\cdots f_*
  g^* \cdots$ of sheaves; we consider natural transformations~$\phi
  \colon F \to G$ contained in the smallest class closed under the
  inclusion of those of the following three types:
  \begin{description}
  \item[Functorial] \mbox{}\nopagebreak
    \begin{itemize}
    \item The identity transformations;
    \item Functoriality isomorphisms~$f_* g_* \cong (fg)_*$ and~$g^*
      f^* \cong (fg)^*$;
    \item Functoriality isomorphisms~$\id_* \cong \id$ and $\id^*
      \cong \id$;
    \item Compositions of transformations;
    \item Applications of~$f_*$ or~$f^*$ to either side of a
      transformation;
    \end{itemize}
  \item[Adjunction] \mbox{}\nopagebreak
    \begin{itemize}
    \item Transformations corresponding under adjunction of~$^*$
      and~$_*$, or equivalently, the units~$\id \to f_* f^*$ and
      counits~$f^* f_* \to \id$ of such adjunctions;
    \end{itemize}
  \item[Inverse] \mbox{}\nopagebreak
    \begin{itemize}
    \item The inverses of all invertible base change
      transformations~$g^* f_* \isoarrow \tilde{f}_* \tilde{g}^*$,
      as in diagram~\ref{eq:base change diagram}.
    \item The inverses of all invertible transformations~$f_* \to f_*
      g_* g^*$ or~$f^* \to g_* g^* f^*$ derived from adjunction units.
    \end{itemize}
  \end{description}
\end{theorem*}

To illustrate the functors in question, one such is by definition
equivalent to the data of a zigzag diagram of spaces
\begin{equation}
  \label{eq:SGF zigzag diagram}
  \includestandalone{img44}
\end{equation}
corresponding to the functor~$(g_{1{*}} \cdots\allowbreak g_{m{*}})
\cdots (f_n^* \cdots\allowbreak f_1^*)$, where the central ellipsis
(``$\cdots$'') corresponds to further zigzags in the peak ellipsis of
the picture.  This notation is intended to encompass the many variants
such as~$f^* g_*$ or~$f_n^* \cdots f_1^*$ by omitting some of the
terms from either side of the composition (respectively, maps from the
diagram).

Each of these functors has an associated ``roof'' (\ref{roof}),
depicted diagramatically as folows:
\begin{subequations}
  \allowdisplaybreaks
  \label{eq:zigzag roof}
  \begin{gather}
    \label{eq:zigzag diagram}
    F =
    \includestandalone{img45a} \\
    \label{eq:roof diagram}
    \roof(F) = \includestandalone{img45b} = a_{F{*}} b_F^*
  \end{gather}
\end{subequations}

Given the above context, our main theorems mostly claim the following:

\begin{theorem*}{thm}
  A transformation $\phi \colon F \to G$ as above is unique if either:
  \begin{enumerate}
  \item \label{en:formal} $F = G$ is of the form~$f_*$ or~$g^*$
    and~$\phi$ draws from the functorial and adjunction
    transformations and inverse base changes; or, if only~$G$ is of
    the form~$f_* g^*$ and~$\phi$ draws only from the functorial
    transformations, base changes,~and inverse base changes.
  \item \label{en:geometric} $\phi$ is arbitrary, if~$G$ is
    isomorphic to its roof~$a_{G*} b_G^*$, if the pair of roof
    maps~$(b_G, a_G)$ of~\ref{eq:roof diagram} factors through the
    pair~$(b_F, a_F)$, and if the latter is an immersion into~$A
    \times C$ (in the notation of the diagram).
  \end{enumerate}
\end{theorem*}

The potential isomorphism mentioned in~\ref{en:geometric} is
canonically defined in \ref{roof}; it is effectively the sequence of
base changes corresponding to~\ref{eq:roof diagram}.  We have
strengthened the hypotheses unnecessarily to make it more
straightforward.  See the example of ``cohomological pullback'' for a
discussion.

Functors and transformations of this nature are found throughout
geometry.  Part~\ref{en:formal} is the more common application, and
signifies that a diagram can be shown to commute by simple
manipulation.  Part~\ref{en:geometric} indicates at least a slightly
domain-specific computation that (as its proof will eventually show)
is not \emph{entirely} symbolic manipulation.

\subsection*{Coherence of tensor products.}

Recall that the \define{tensor product} of sheaves~$\sh{F}$
and~$\sh{G}$ on a scheme~$X$ satisfies the relations
\begin{equation*}
  \sh{F} \otimes \sh{G}
  = \Delta^* (\sh{F} \boxtimes \sh{G}),
  \qquad
  \sh{F} \boxtimes \sh{G}
  = \on{pr}_1^* \sh{F} \otimes \on{pr}_2^* \sh{G},
\end{equation*}
where~$\on{pr}_{1,2}$ are the projections~$X \times X \to X$
and~$\Delta \colon X \to X \times X$ is the diagonal morphism.
Taking the latter, ``outer'' tensor product as the fundamental
object allows the convenient formulation of tensor product
identities entirely in terms of the functors described by
the~theorem.  We obtain the commutativity and associativity
constraints,
\begin{align*}
  \sh{F} \otimes \sh{G} \cong \sh{G} \otimes \sh{F} &&
  (\sh{F} \otimes \sh{G}) \otimes \sh{H} \cong \sh{F} \otimes
  (\sh{G} \otimes \sh{H}),
\end{align*}
using the functoriality of $^*$ on the analogous identities:
\begin{align*}
  \Delta = \on{sw} \Delta &&
  (\Delta \times \id) \Delta = (\id \times \Delta)\Delta,
\end{align*}
where~$\on{sw} \colon X \times X \to X \times X$ is the coordinate
swap.

As an easy consequence, we obtain the conclusion of Mac Lane's
coherence theorem for the tensor category of sheaves: any natural
transformations of two parenthesized multiple tensor products
constructed only from commutativity and associativity constraints
are equal.  Indeed, all such parenthesized products are repeated
pullbacks~$g^*$ for various maps~$g$, so such a transformation is a
map
\begin{equation*}
  g^* \cong g_1^* \cdots g_n^* \to h_1^* \cdots h_n^* \cong h^*,
\end{equation*}
where the outside isomorphisms are by functoriality of pullback.
Therefore part~\ref{en:formal} of the~theorem applies.

\subsection*{Projection formula and compatibility diagrams.}
As a more interesting example, we consider the \define{projection
  formula} morphism for a map~$f \colon X \to Y$ and
sheaves~$\sh{F}$ and~$\sh{G}$ on~$X$ and~$Y$ respectively:
\begin{equation*}
  f_* \sh{F} \otimes \sh{G} \to f_*(\sh{F} \otimes f^* \sh{G}).
\end{equation*}
It can be expediently defined by first forming the cartesian diagram
\begin{equation*}
  \includestandalone{img46}
\end{equation*}
and then rewriting the projection formula as
\begin{equation*}
  \Delta_Y^*(f \times \id)_*(\sh{F} \boxtimes \sh{G})
  \to f_* \Gamma_f^*(\sh{F} \boxtimes \sh{G})
\end{equation*}
and realizing it as a base change morphism.  Here we have used the
fact that~$f_* \sh{F} \boxtimes \sh{G} \cong (f \times \id)_*
(\sh{F} \boxtimes \sh{G})$.

Examples of diagrams of the projection formula are diagrams~(2.1.11)
and~(2.1.12) of~\cite{conrad}, which are correctly said to be
trivial but are nonetheless rather tedious, and in fact
automatically commute.  For verification we reproduce them here,
using our notation.  The first one is:
\begin{equation*}
  \tag{2.1.11}
  \includestandalone{img47}
\end{equation*}
where the horizontal maps are associativity, the lower-left vertical
map is the isomorphism~$f^*(\sh{G} \otimes \sh{H}) \cong f^* \sh{G}
\otimes f^* \sh{H}$ that is easily deduced from the outer product
formulation,~and the others are projection formula maps.  Here, we
have~$f \colon Y \to X$ for schemes $X$~and~$Y$, where~$\sh{F}$ is a
sheaf on~$X$ and $\sh{G}$~and~$\sh{H}$ are on~$Y$.  If we
write $\Delta_X$~and~$\Delta_Y$ for the respective diagonal
morphisms, then the two directions around the diagram are
transformations
\begin{equation*}
  \Delta_X^* (f \times \id)_* (\id \times \Delta_Y)^*
  \to
  f_*\Delta_Y^* (\id \times f)^* (\Delta_Y \times \id)^* ((\id
  \times f) \times \id)^*,
\end{equation*}
and the latter is of the form~$f_* g^*$ after applying functoriality
isomorphisms to the chain of pullbacks, so part~\ref{en:formal} of
the~theorem applies.

The second diagram is:
\begin{equation*}
  \tag{2.1.12}
  \includestandalone{img48}
\end{equation*}
where the upper and lower horizontal and lower-left vertical maps
are projection formulas, the other vertical maps are functoriality,
and the middle map is defined by inverting either the upper-left or
lower-right vertical maps and composing; to check the small squares
it suffices to check the big one.  Here, we have~$f \colon Z \to Y$
and~$g \colon Y \to X$, with sheaves~$\sh{F}$ on~$Z$ and~$\sh{G}$
on~$X$, and the two ways around the diagram are transformations
\begin{equation*}
  \Delta_X^* (gf \times \id)_*
  \to
  g_* f_* \Delta_Z^* (\id \times f)^*(\id \times g)^*,
\end{equation*}
which again can be brought into the form treated by
part~\ref{en:formal} of the~theorem by condensing the pullbacks and
pushforwards.

\subsection*{Cohomological pullback.}
An obvious way of involving units of adjunction is to introduce the
\emph{change of space} or \emph{cohomological pullback} maps,
defined as follows.  Let~$*$ denote the base (final) scheme and for
any~$X$ with its canonical map~$p \colon X \to *$, let~$H^*(X,?) =
p_*$ denote global cohomology.  Cohomological pullback is defined
for any map~$f \colon Y \to X$ as the natural transformation
(written with reference to a sheaf~$\sh{F}$ on~$X$)
\begin{equation*}
  H^*(X,\sh{F}) = p_* \sh{F} \to p_* f_* f^* \sh{F} \cong (pf)_* f^*
  \sh{F} = H^*(Y,f^* \sh{F})
\end{equation*}
since~$pf$ is the structure map for~$X$.  It follows easily from
part~\ref{en:formal} of the~theorem that any composition of
cohomological pullbacks from~$X$ to~$Z$ connected by some number of
composable maps is independent of which intermediate maps are used;
i.e.~that pullback is functorial in the same sense as~$^*$ itself.
Indeed, pullback is a transformation of the form~$p_{X{*}} \to
p_{Z{*}} f^*$, and by applying ``reverse natural adjunction''
(see~\ref{eq:RNA map} for this not-widely-mentioned operation) is
equivalent to a transformation~$p_{X{*}} f_* \to p_{Z{*}}$, which by
a functoriality isomorphism is equivalent to a transformation~$(p_X
f)_* \to p_{Z{*}}$, which is unique.

More interestingly, let us say that~$f$ is a \emph{cohomological
  isomorphism} if its associated pullback is an isomorphism.  Then
we can compose pullbacks and the~\emph{inverses} of pullbacks that
are isomorphisms to get maps between the global cohomology of spaces
that are \emph{not} connected by composable maps, but merely zigzags
(similar to~\ref{eq:zigzag diagram}) where the zigs are invertible.
Nonetheless, these are transformations to which
part~\ref{en:geometric} of the~theorem applies.  For example, if we
use the exact diagram shown there, we get transformations~$p_{X{*}}
\to p_{Z{*}}$.

The criterion given by the~theorem as stated above for these
transformations to be unique is somewhat restrictive: we would require
that there be an actual map~$f \colon Z \to X$ commuting with the
structure maps~$p$, in which case the zigzag pullback coincides with
the actual pullback along $f$.  A slightly less restrictive hypothesis
is given in~\ref{main theorem}: that the projection map~$X \times Z
\to Z$ be a cohomological isomorphism; then the universal
representative of zigzag pullbacks from~$X$ to~$Z$ passes through~$X
\times Z$.

\subsection*{Cup products and cohomology algebras.}
A final application of the~theorem is again to tensor products: the
\define{cup product} on cohomology, defined in the following
two-step manner.  For any sheaf~$\sh{F}$ on~$X$, writing
just~$H^*(\sh{F})$ rather than~$H^*(X,\sh{F})$, we have a map
\begin{equation*}
  H^*(\sh{F}) \otimes H^*(\sh{F}) \to H^*(\sh{F} \otimes \sh{F})
\end{equation*}
defined as the cohomological pullback along~$\Delta \colon X \to X
\times X$
\begin{equation*}
  (p \times p)_* (\sh{F} \boxtimes \sh{F})
  \to
  (p \times p)_* \Delta_* \Delta^* (\sh{F} \boxtimes \sh{F})
  \cong p_* (\sh{F} \otimes \sh{F})
\end{equation*}
We have used the fact that~$(p \times p)_*(\sh{F} \boxtimes \sh{F})
\cong p_* \sh{F} \boxtimes p_* \sh{F} = p_* \sh{F} \otimes p_*
\sh{F}$.  Then, if~$\sh{F}$ is a ring sheaf with multiplication~$m
\colon \sh{F} \otimes \sh{F} \to \sh{F}$, we compose with the
induced map~$H^*(m) \colon H^*(\sh{F} \otimes \sh{F}) \to
H^*(\sh{F})$ to obtain the cup product
\begin{equation*}
  \smile \colon H^*(\sh{F}) \otimes H^*(\sh{F}) \to H^*(\sh{F}).
\end{equation*}

It is easy to show, from the functoriality of cohomological
pullback, that whenever we have a map~$f \colon X \to Y$ and thus we
have a commutative diagram of spaces, we get one of natural
transformations:
\begin{align*}
  \includestandalone{img49a} && \includestandalone{img49b}
\end{align*}
Indeed, we need only adjust the second diagram to bring the ring
multiplication map of~$\sh{F}$ out.  But this follows from
naturality, i.e.~the left square in the diagram below commutes:
\begin{equation*}
  \includestandalone{img50}
\end{equation*}
The rest of the proof is just that the two paths in the right half
of the diagram are a pair of alternative cohomological pullbacks
from~$X \times X$ to~$Y$, hence equal.

By a similar token, we can show that the cup product is an algebra
structure on~$H^*(\sh{F})$.  We will gloss over the details that the
ring multiplication map can be factored out of the diagrams
expressing commutativity, unitarity,~and associativity, after which
they all express aspects of the functorial nature of cohomological
pullback again.

We can go further in the case that~$X = G$ is a group scheme, whose
multiplication map induces a \emph{coalgebra} structure
on~$H^*(G,\sh{F})$ for any sheaf~$\sh{F}$, with comultiplication
defined by cohomological pullback along~$G \times G \to G$ and
counit being pullback along~$* \to G$.  Once again, functoriality of
the pullback proves that this is indeed a coalgebra.  Taken together
with the cup product we easily combine the structures into a Hopf
algebra, since (modulo writing the naturality diagrams that
pull~$H^*(m)$ out) all the compatibilities express equality of
various iterated cohomological pullbacks.

This is standard algebraic topology material and nothing in this
example has actually cited the main theorem directly; we include it
only to note that it rests on a fact, functoriality of cohomological
pullbacks, that is a corollary of our main theorem, as well as to
indicate the relevance of this setup to well-known constructions.

\section{Fundamentals of string diagrams}
\label{s:string diagram intro}

Here is an overview of the use of pictorial notation for
category-theoretic computations, with an emphasis on applications to
this paper, together with proofs of the diagram equivalences presented
in \ref{s:string diagrams}.

String diagrams are used to depict the algebra of category theory
visually, and there appears to be a variety of styles in which they
are drawn.  These topological representations of natural
transformations (introduced by Penrose~\citelist{\cite{penrose1}
  \cite{penrose2}} and whose first introduction into the pure
mathematics literature seems to have been in~\cite{joyal-street}; an
instructive introduction is given by the video series~\cite{catsters})
seem to have the same mysterious effectiveness in category theory that
the Leibniz notation does in calculus, often indicating algebraic
truths through visual intuition.  It does not appear that the
topological data is intrinsically important, though some experience
with string diagrams suggests that removing loops is an essential
first step in reducing them.

\subsection*{Overview of string diagrams.}
String diagrams are a notational paradigm for representing natural
transformations and, in particular, for making the following concepts
from category theory more amenable to intuitive manipulation.

\subsection*{Notation.}

Functors and natural transformations can be combined in composition in
several ways, reflecting the bicategorical structure of the category
of small categories. Two functors may of course be composed, which we
denote by pure juxtaposition without the symbol~$\circ$. If~$\phi
\colon A \to B$ is a natural transformation of two functors
$A$~and~$B$, and if~$F$ is a functor such that $FA$~and~$FB$ are
defined, then we write~$F\phi$ to mean the associated natural
transformation~$FA \to FB$ such that~$F\phi_x = F(\phi_x)$ for any
object~$x$.  Likewise, if $AF$~and~$BF$ are defined, then we
write~$\phi F$ to mean the functor~$AF \to BF$ such that~$\phi F_x =
\phi_{F(x)}$.

We use the common term ``horizontal composition'' of natural
transformations~$\phi \colon A \to B$ and~$\psi \colon C \to D$ to
mean the following natural transformation~$AC \to BD$, which we will
write as simply~$\phi\psi$ without ambiguity:
\begin{equation}
  \label{eq:horizontal composition}
  \includestandalone{img51}
\end{equation}

\subsection*{Introduction to string diagrams.}

Every string diagram depicts a single natural transformation of
functors, where all compositions (of both functors and
transformations) are made explicit.  Its main feature is a web of
continuous paths recording the history of each functor as it is
transformed; these transformations occur at the intersections.  The
connected components of the complement of this web represent the
categories related by these functors, with the direction of
composition being the same as in any diagram $X \to Y \to Z$.  Thus,
each functor is like a river flowing between two banks, and to apply
the functor is to cross the river.  The intersections of paths are
natural transformations, with the direction of composition being
upwards.  In this paper, the allowable intersections are those given
in~\ref{fig:basic SGNTs} representing as marked the basic SGNTs
of~\ref{basic SGNTs}.  We give a precise definition of the string
diagrams we use:

\begin{theorem}{defn}{string diagram}
  A \define{string diagram} is a planar graph that is the union of
  shapes of the general form shown in~\ref{fig:basic SGNTs}, where the
  various segments may have any lengths.  These shapes may only
  intersect at their ends, and conversely, any end of a shape in the
  diagram must either join another shape, or be continued upward or
  downward to an infinite vertical ray.  The vertices of the graph are
  either the univalent caps of the~$\triv^*$ or~$\triv_*$ shapes, the
  trivalent intersections of the~$\comp^*$ or~$\comp^*$ shapes, or the
  right-angle turns of the various shapes; the edges are the connected
  line segments in the complement of the vertices.  A string diagram
  may also have bi-infinite, directed vertical lines disjoint from the
  other portions.
\end{theorem}

We write our diagrams with their edges doubled.  This has no practical
meaning, but aside from making the diagrams more aesthetic, it
provides some topological intuition that will be explained
later. There is a correspondence between SGNTs and string diagrams
labeled as in \ref{fig:basic SGNTs}.

\begin{theorem}{defn}{string diagram transformation}
  Let~$D$ be a string diagram with its edges labeled; then any
  horizontal slice not containing a vertex of~$D$ determines an SGF by
  composing the labels left-to-right (rather than right-to-left, as is
  traditional in algebraic notation).  In particular, its eventual
  slices in the two vertical directions determine two~SGFs
  $F$~and~$G$.  We describe how to obtain an~SGNT $\phi_D \colon F \to
  G$; it satisfies the following properties:
  \begin{itemize}
  \item Each of the diagrams of~\ref{fig:basic SGNTs} has the
    interpretation as an~SGNT given there.
  \item If~$D$ is split by any simple path extending to horizontal
    infinity in both directions and not containing any vertices or
    crossing any horizontal edge (the latter, a~restriction born of
    our choice of visual style), then both the lower and upper parts
    $D_1$~and~$D_2$ are string diagrams (after extending their cut
    edges to infinity).  We have $\phi_D = \phi_{D_2} \phi_{D_1}$.
    That is, vertical composition of~SGNTs is vertical concatenation
    of diagrams.
  \item If the complement of~$D$ contains a simple path extending to
    vertical infinity in both directions, then its left and right
    parts $D_1$~and~$D_2$ are string diagrams and~$\phi_D$ is the
    horizontal composition of~$\phi_{D_2}$ after~$\phi_{D_1}$.  That
    is, horizontal composition of~SGNTs is horizontal concatenation of
    diagrams.
  \end{itemize}
  Under this interpretation, we have the following correspondence:
  \begin{itemize}
  \item The complement of~$D$ in~$\R^2$ consists of finitely many
    connected open sets, each of which represents a category.
  \item Each vertical edge is a functor from the category on its left
    to the one on its right; an upward edge is an~$f_*$, and a
    downward edge is an~$f^*$ (thus, the direction of~$f$ is
    always left-to-right when facing along the edge).
  \item A natural transformation occurs at any horizontal edge
    (including those of zero length at the valence-$1$ vertices).  Its
    direction is from the functor below it to the functor above.
  \end{itemize}
  We will say that two diagrams are \define{equivalent} if they define
  equal transformations.
\end{theorem}

\subsection*{Proofs of string diagram identities.}

In this subsection we collect the proofs from \ref{s:string diagrams},
which are by and large trivial translations of symbolic
category-theory notation into string diagrams.

\begin{proof}[Proof of \ref{fig:adjunctions}]
  These diagrams correspond to the fundamental identities
  \ref{eq:adjoint identities} of the unit and counit that ensure that
  they define an adjunction.
\end{proof}

\begin{proof}[Proof of \ref{fig:compositions (assoc)}]
  The left side is the string diagram equivalent
  of~\ref{eq:compositions assoc}.  The right side follows by Yoneda's
  lemma from~\ref{eq:composition adjunction} and the left side.
\end{proof}

\begin{proof}[Proof of \ref{fig:compositions (inv)}]
  The left two diagrams express in string diagram language
  that $\comp_*$ comes in inverse pairs, which we have assumed by
  definition.  The right two also express this, for~$\comp^*$, which
  is~\emph{not} a definition, but follows immediately
  from~\ref{eq:composition adjunction} and Yoneda's lemma.
\end{proof}

\begin{proof}[Proof of \ref{fig:trivializations triv}]
  The first and third follow from the definition, since~$\triv_*$
  comes by hypothesis in an inverse pair.  The second and fourth
  express the same for~$\triv^*$ and follow
  from~\ref{eq:trivialization adjunction} by Yoneda's lemma.
\end{proof}

\begin{proof}[Proof of \ref{fig:trivializations adj}]
  The third and fourth express the
  correspondence~\ref{eq:trivialization adjunction} of~$\triv^*$
  and $\triv_*$ via Yoneda's lemma.  The first and second express the
  same correspondence of their inverses.
\end{proof}

\begin{proof}[Proof of \ref{fig:trivializations comp}]
  The entire first row of equalities is the string diagram version
  of~\ref{eq:trivializations comp}, which we assume to hold.  The
  second row follows from this equation and from~\ref{eq:composition
    adjunction} and~\ref{eq:trivialization adjunction} by Yoneda's
  lemma.
\end{proof}

\begin{proof}[Proof of \ref{fig:adjunction-composition 1}]
  These correspond to the assumed compatibility of~$\comp_*$
  with $\comp^*$ via adjunction expressed in~\ref{eq:composition
    adjunction}.  This compatibility entails the equality of the two
  units, respectively the two counits, of the adjunctions for the two
  sides of $(fg)^* = g^* f^*$ (resp.~of $(fg)_* = f_* g_*$):
  \begin{align}
    \label{eq:composition units counits}
    \includestandalone{img52a} && \includestandalone{img52b}
  \end{align}
  If we apply the second row of~\ref{fig:compositions (inv)} to each
  of the four composition shapes in these two diagrams, we get the
  four diagrams of~\ref{fig:adjunction-composition 1}.
\end{proof}

\begin{proof}[Proof of \ref{fig:adjunction-composition 2}]
  We illustrate the upper-left one; the others are formally identical.
  \begin{equation}
    \label{eq:adjunction-composition 2}
    \includestandalone{img53}
  \end{equation}
  using first~\ref{fig:adjunctions} and
  then~\ref{fig:adjunction-composition 1}.
\end{proof}

\begin{proof}[Proof of \ref{fig:compositions (nontrivial)}]
  This interesting diagram is easy to prove from the previous ones.
  We will omit the labels, since they can be inferred by comparison
  with the figure.  In the equalities below, matching blue and red
  indicate which portions of the diagram are changed; violet indicates
  a shape that is both blue and red.
  \begin{equation}
    \label{eq:compositions (nontrivial)}
    \includestandalone{img54}
  \end{equation}
  using the second row of~\ref{fig:compositions (inv)},
  then~\ref{fig:compositions (assoc)}, then the first row
  of~\ref{fig:compositions (inv)}.  The proof of the~$^*$ version is
  identical with the appropriate change of notation (or follows by
  Yoneda's lemma from~\ref{eq:composition adjunction}).
\end{proof}

\begin{proof}[Proof of \ref{SD loop}]
  Again using colors to indicate the affected regions, we have
  \begin{equation}
    \includestandalone{img55}
  \end{equation}
  by~\ref{fig:compositions (nontrivial)} (twice) and
  then~\ref{fig:compositions (inv)}.
\end{proof}

\subsection*{Topological intuition and application to objects.}
It is a tautology of string diagrams that they are useful because they
are visual; our chosen notation, through a combination of design and
good fortune, happens to display a remarkable compatibility with
topological intuition that has been commented on several times earlier
in this paper.  

The general principle of our string diagrams is that topologically
identical diagrams are equivalent.  We don't know how to prove this in
that kind of generality, though a close examination of our many
figures will show that they are all topological trivialities (i.e.\
diagrams equivalent through ambient isotopy are equal as natural
transformations), with the exception of
Lemmas~\plainref{fig:compositions (inv)}
and~\plainref{fig:trivializations triv}, if one takes the double-line
notation seriously (this is the reason for its invention).  In fact,
if in the first line of the latter figure one removes the inner loop
formed by one of the sets of doubled lines, the identity is again
topologically accurate.  However, it is easy to give a simple version
that is almost too subtle to remark upon.  Indeed, we have never
invoked it directly, but it is used constantly to rearrange diagrams
to our liking.

\begin{theorem}{lem}{diagrams translation}
  Let $D \subset \R^2$ be a string diagram and let $f \colon \R^2 \to
  \R^2$ be any continuous map whose restriction to each edge of~$D$ is
  an affine map with positive scale factor in the direction
  (horizontal or vertical) of that edge.  Then~$f(D)$ is an equivalent
  string diagram to~$D$.
\end{theorem}

\begin{proof}
  We see, first, that~$f(D)$ is a string diagram, since each basic
  shape is transformed by~$f$ into a basic shape of the same type, and
  any edge abuts the same basic shapes or is infinite in both diagrams
  by continuity and positive affinity of~$f$.  Furthermore,~$f$
  induces a bijection between components of $\R^2 \setminus D$, and
  any two adjacent components are separated by an edge of the same
  direction; thus, the interpretation of each edge as a map of schemes
  remains valid.  To see that the diagrams are equivalent, first
  consider two distinct connected components of~$D$, if they exist.
  By a simple induction, each of them is equivalent, as a string
  diagram, to its image under~$f$, and these images are translated
  with respect to each other as compared to the originals.  But
  translational motion is an equivalence by naturality.
\end{proof}

Ours is not the only notation for string diagrams; in fact, it appears
to differ in some particulars from others commonly used.  McCurdy and
Melli\`es~\citelist{\cite{mccurdy} \cite{mellies}}, following Cockett
and Seely~\cite{cockett-seely}, use ``functorial boxes'' rather than
line segments to denote functors, but this reflects their different
application of the string diagram visualization: their diagrams denote
objects of a monoidal category and their morphisms, rather than
functors and their natural transformations.  In fact, our notation is
more general and can easily be extended to display the specific
instances of functors and natural transformations obtained by
evaluating them at individual objects of the underlying categories.
However, we believe that string diagram notation is inherently
domain-specific, depending on the natural properties and relationships
of the functors and transformations it describes, so it is unlikely
that there is one single, universally effective style.  That said, the
box notation is entirely free of features that display the nature of
the functors that appear in it.

As an example of the two diagram styles, we present the definition of
a monoidal functor and a monoidal natural transformation between two
monoidal functors.  In the notation of Melli\`es (which is
particularly attractive), the monoidality of a functor~$f^*$ is
expressed in \ref{eq:mellies monoidal functor}.
\begin{figure}[ht]
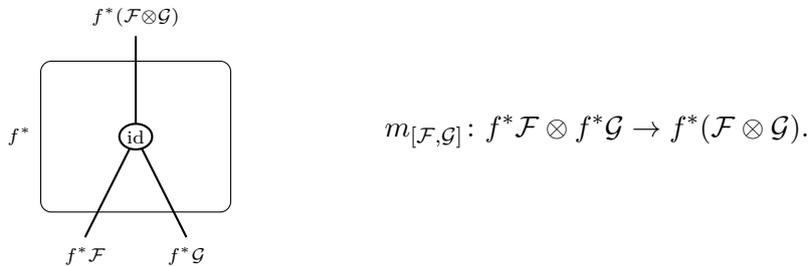

  \begin{align*}
    \includestandalone{img56} &&
    m_{[\sh{F}, \sh{G}]} \colon
    f^* \sh{F} \otimes f^* \sh{G} \to f^*(\sh{F} \otimes \sh{G}).
  \end{align*}
  \caption{Monoidal functor \`a la Melli\`es}
  \label{eq:mellies monoidal functor} 
\end{figure}
The meaning of the box is that, inside, the~$f^*$ is ``stripped'' from
the objects; the name of the map~$m_{[\sh{F}, \sh{G}]}$ is not
mentioned, as it is implicit to the diagram, which is specialized to
the needs of monoidal categories. By comparision, the monoidality
of~$f^*$ may be expressed in our notation as in \ref{eq:our monoidal
  functor}.
\begin{figure}[ht]
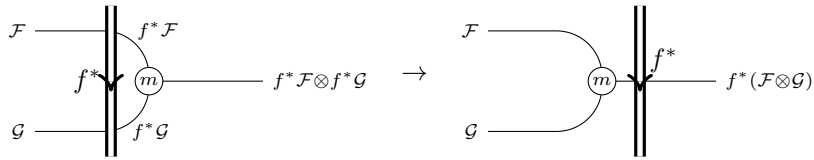

  \begin{equation*}
    \includestandalone{img57}
  \end{equation*}
  \caption{Monoidal functor in our notation}
  \label{eq:our monoidal functor}
\end{figure}
In each of its diagrams, the background layer is a horizontal string
diagram with the strings being objects of the category corresponding
to the planar region containing them (according to \ref{string diagram
  transformation}) and their intersections being transformations from
left to right (here, marked, as we have no specialized notation for
monoidal categories as we do for geofibered ones).  Vertical stacking
of objects' strings corresponds to their monoidal product.

Unlike in the purely functorial diagrams, each individual diagram does
not depict a morphism~$m$; rather, it denotes the endpoints of such a
morphism (the rightmost labeled object in each) and the ``location''
of~$m$.  The morphism itself is obtained by ``sweeping'' the
functorial diagram over the objects from left to right, as shown
above.  Whenever it crosses a marked morphism, the motion corresponds
to the introduction of that map between the ``endpoint'' on the left
and the one on the right.  (Obviously, we must forbid endpoint
diagrams that would allow degenerate or ambiguous configurations.)

Thus, the diagrams display the history of an object and the
\emph{potential} natural morphisms, but it is the topological
relationships between them that give the actual maps.

As a further example of this, consider the monoidality of the natural
transformation~$\theta=\comp^*(g,f)$.  In Melli\`es' notation, it is
the equation of diagrams in \ref{eq:mellies monoidal transformation}.
\begin{figure}[ht]
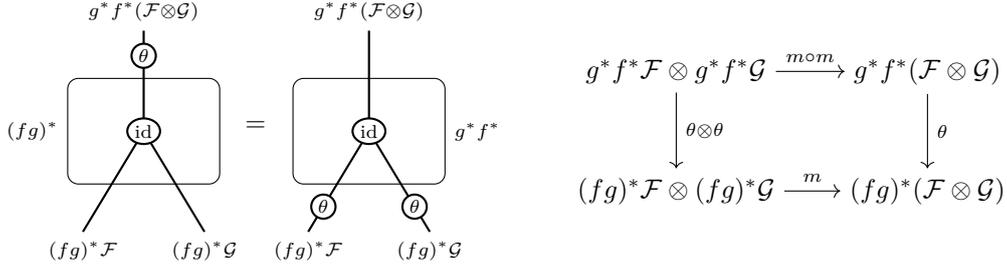

  \begin{align*}
    \includestandalone{img58a} && \includestandalone{img58b}
  \end{align*}
  \caption{Monoidal transformation \`a la Melli\`es}
  \label{eq:mellies monoidal transformation}
\end{figure}
The interpretation is straightforward; vertical juxtaposition of
morphisms is composition and horizontal is monoidal product.  In our
notation, we express it in \ref{eq:our monoidal transformation} as a
closed loop of translations.
\begin{figure}[ht]
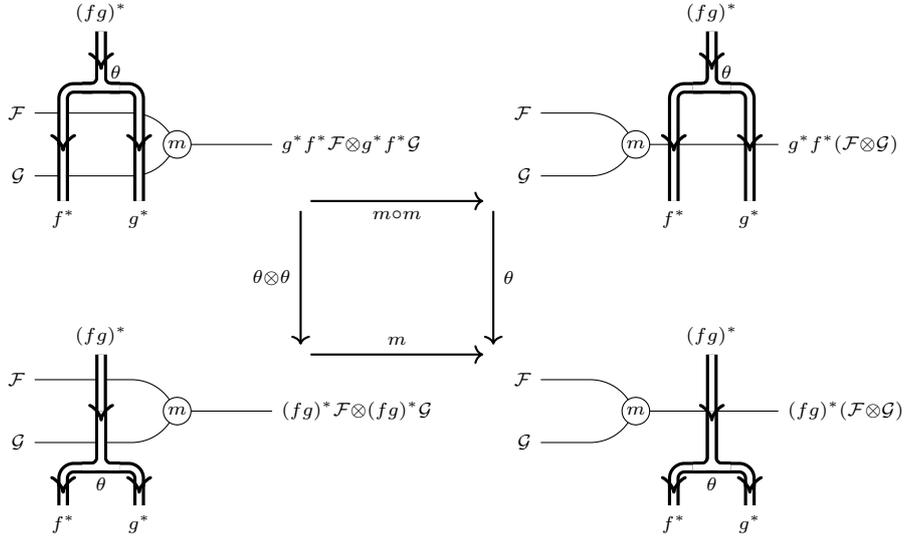

  \begin{equation*}
    \includestandalone{img59}
  \end{equation*}
  \caption{Monoidal transformation in our notation} 
  \label{eq:our monoidal transformation} 
\end{figure}
Moving the functorial ``foreground'' from left to right, as before,
produces one or two iterations of~$m$, depending on how many of its
strings cross that vertex.  Moving the objects' ``background'' from
bottom to top produces an instance of the natural
transformation~$\comp^*(f,g)$ (the one on the object~$\sh{F} \otimes
\sh{G}$ for the~$\theta$ edge; for each factor of~$\theta \otimes
\theta$, we get the instances on $\sh{F}$~or~$\sh{G}$ individually).
Since the translational motions commute, so does the diagram of maps.
We cannot help but feel that this is an instance of the
Eckmann--Hilton argument for the commutativity of higher homotopy
groups (as depicted in Hatcher's book~\cite{hatcher}*{p.~340}).

The fact that we have expressed a single natural morphism, as
in~\ref{eq:our monoidal functor}, in terms of an equivalence of
several diagrams, would seem to be at odds with our intention of using
individual diagrams to express natural transformations.  In fact, this
morphism could be expressed in a purely functorial notation with some
additional convention as to the appearance of the ``monoidal product
string'', which would involve the complication of expressing its
bivariant nature.  However, it is actually \emph{this} approach that
is at odds with the philosophy of our diagrams, since the morphism
of~\ref*{eq:our monoidal functor} is not a ``structural'' morphism but
rather a ``compatibility'' between two independent structures.
Moreover, the structures are not of an equivalent nature: the monoidal
product is ``internal'' and the pullback functor (or any monoidal
functor) is ``external''.  So it is entirely appropriate that their
compatibility should be expressed through superimposed string diagrams
of which one represents objects in the categories cut out by the
other.

In summary: in this paper we have used ``string diagrams''
representing functors and their natural transformations, but we can
``specialize'' them to the corresponding values and their natural maps
when applied to specific objects by superimposing such a diagram on a
horizontal diagram of objects and morphisms in the categories cut out
by the functorial diagram.  Then any continuous motion (avoiding some
degeneracies) ``instantiates'' the morphisms when the strings of one
diagram cross the vertices of the other, and any closed loop should
represent a commutative diagram of these maps.

We feel that this topological intuition is a valuable asset in a
graphical algebraic notation, since it elevates the picture from
simply a two-dimensional symbolic calculus to a genuinely graphical
reasoning tool.

\end{document}